\documentclass[a4paper,11pt,reqno]{amsart}

\usepackage[headings]{fullpage}

\usepackage{mathtools}
\usepackage{dsfont,amssymb}
\usepackage{tensor}
\usepackage{tikz}
\usetikzlibrary{cd}
\usepackage[all]{xy}
\usepackage{xcolor}
\usepackage{pxfonts}

\usepackage{hyperref}

\usepackage[nobysame,alphabetic,initials,msc-links]{amsrefs}

\DefineSimpleKey{bib}{how}

\renewcommand{\eprint}[1]{#1}
\BibSpec{misc}{%
  +{}{\PrintAuthors}  {author}
  +{,}{ \textit}      {title}
  +{,}{ }             {how}
  +{}{ \parenthesize} {date}
  +{,} { available at \eprint}        {eprint}
  +{,}{ available at \url}{url}
  +{,}{ }             {note}
  +{.}{}              {transition}
}

\mathchardef\mhyph="2D

\numberwithin{equation}{section}

\newtheorem{theorem}{Theorem}[section]
\newtheorem{corollary}[theorem]{Corollary}
\newtheorem{lemma}[theorem]{Lemma}
\newtheorem{proposition}[theorem]{Proposition}
\theoremstyle{remark}
\newtheorem{remark}[theorem]{Remark}

\theoremstyle{definition}
\newtheorem{definition}[theorem]{Definition}

\def\-{{\hbox{-}}}
\newtheorem{theoremAlph}{Theorem}

\newcommand\bp{\begin{proof}}
\newcommand\ep{\end{proof}}

\newcommand{\C}{\mathbb C}

\newcommand{\bA}{\mathbb A}
\newcommand{\bE}{\mathbb E}

\newcommand{\bK}{\mathbb K}
\newcommand{\cA}{\mathcal A}

\newcommand{\bB}{\mathbb B}
\newcommand{\bD}{\mathbb D}

\newcommand{\cC}{\mathcal C}
\newcommand{\cD}{\mathcal D}
\newcommand{\cE}{\mathcal E}
\newcommand{\cF}{\mathcal F}

\newcommand{\cK}{\mathcal K}

\newcommand{\cL}{\mathcal L}
\newcommand{\cM}{\mathcal M}
\newcommand{\cN}{\mathcal N}
\newcommand{\cO}{\mathcal O}

\newcommand{\cR}{\mathcal R}
\newcommand{\cS}{\mathcal S}

\newcommand{\cU}{\mathcal U}

\newcommand{\cY}{\mathcal Y}
\newcommand{\cZ}{\mathcal Z}

\newcommand{\mpcC}{\cC^{mp}}

\newcommand{\fgpBim}{{\sf Bim_{fgp}}}

\newcommand{\fdRep}{{\sf Rep_{fd}}}

\DeclareMathOperator{\WAlg}{W*Alg}
\DeclareMathOperator{\End}{End}
\DeclareMathOperator{\Hom}{Hom}

\DeclareMathOperator{\Irr}{Irr}

\DeclareMathOperator{\Mod}{Mod}

\DeclareMathOperator{\Rep}{Rep}

\DeclareMathOperator{\Vect}{Vec}

\DeclareMathOperator{\Bal}{Bal}

\DeclareMathOperator{\tr}{tr}

\DeclareMathOperator{\Bim}{Bim}

\DeclareMathOperator{\For}{For}
\DeclareMathOperator{\RMod}{RMod}

\DeclareMathOperator{\LMod}{LMod}
\DeclareMathOperator{\Surf}{Surf}
\DeclareMathOperator{\CCat}{C*Cat}

\DeclareMathOperator{\Corr}{Corr}
\DeclareMathOperator{\Disk}{Disk}

\newcommand\op{\text{op}}
\newcommand\un{{\mathds 1}}

\newcommand{\id}{\mathrm{id}}

\newcommand{\Hilb}{\mathrm{Hilb}}

\newcommand{\YD}{\mathcal{YD}}

\begin{document}




\title{Monadic reconstruction of unitary Drinfeld centers and Factorization Homology}

\date{}

\author{Lucas Hataishi}
\address{}
\email{lucas.yudihataishi@maths.ox.ac.uk}

\begin{abstract}
We prove that the unitary Drinfeld center of a unitary tensor category is equivalente to the category of unitary bimodules for the canonical W$^*$-algebra object, generalizing Müger's result to the non-fusion case. This is then used to express factorization homology in terms of C$^*$-algebraic extensions of symmetric enveloping algebras and actions of Drinfeld dobules of compact quantum groups.
\end{abstract}

\maketitle
\tableofcontents

\section{Introduction}

The Drinfeld center construction, which takes a tensor category and produces a braided tensor category, appears naturally in a variety of contexts in quantum algebra. 
\begin{itemize}
	\item Quantum groups: the Drinfeld center of the representation category of a quantum group is equivalent to the category of representations of its Drinfeld double (\cite{MR0934283}, \cite{MR1381692}). 
	\item Subfactors: Drinfeld centers appear in relation to the symmetric enveloping inclusions associated to subfactors (\cites{MR0996454,MR1302385,MR1316301,MR1332979,MR1782145,MR1742858,MR1966524,MR1966525,MR3509018}).
	\item Condensed-matter physics: Drinfeld centers describe topological charges in Levin-Wen models (\cite{PhysRevB.71.045110}).
	\item Topological Quantum Field Theory (TQFT):  Reshetikhin-Turaev construction from modular tensor categories (\cite{MR1091619}), of which Drinfeld centers of fusion categories are examples.
\end{itemize}

Specially for applications in mathematical-physics, the literature concentrates mostly on fusion categories, finite non-semisimple tensor categories, either rigid or of Grothendieck-Verdier type.  However, the theory of compact quantum groups (\cite{MR3204665}) and the theory of subfactors (\cite{MR0696688,MR1622812,MR3948170,Hataishi_Palomares_2025})gives rich sources of tensor categories with unitary structures and infinite isomorphism classes of simple objects. They are also rigid, which together with unitarity implies semi-simplicity, and have simple units. Such categories are called {\em unitary tensor categories} (UTCs).

The Drinfeld center construction is subtle in the unitary and non-finite context. By naively taking the category $\cZ \cC$ of unitary half-braidings of a UTC $\cC$, one often obtains a very small category. To get an interesting category, first we need to consider unitary half-braidings in the  {\em unitary ind-completion} $\Hilb(\cC)$ of $\cC$. We denote it by $\cZ \Hilb(\cC)$, and call it the {\em unitary Drinfeld center} of $\cC$, see \cite{MR3509018}. $\cZ \Hilb(\cC)$ is then a C$^*$-tensor category with a unitary braiding. The main motivation for this paper is to understand a 2-dimensional fully extended TQFT associated to $\cZ \Hilb(\cC)$, constructed via {\em factorization homology} with values in C$^*$-categories. 

One particular instance, and for us the case of major interest, is when $\cC$ is the category $\fdRep(K_q)$ of unitary representations of the Woronowicz-Drinfeld-Jimbo deformation of a semisimple simply connected Lie group $K$. In this case, writing $\Hilb(\fdRep(K_q)) = \Rep(K_q)$ for simplicity, $\cZ \Rep(K_q)$ coincides with the category $\Rep(DK_q)$ of unitary representations of the Drinfeld double of $K_q$. As advocated in \cite{MR4162277}, the $DK_q$ gives a quantization of the complexification $K_\C$ of $K$. We expect factorization homology with coefficients in $\Rep(DK_q)$ to give a quantization of the algebra of continuous functions on the $K_\C$-spaces
\[
\cR_{K_\C}(\Sigma) := \Hom(\pi_1(\Sigma), K_\C) \ , 
\]
where $\Sigma$ is a compact surface. Geometrically, $\cR_{K_\C}(\Sigma)$ is the moduli space of flat principal $K_\C$-bundles on $\Sigma$, trivialized on a point, and the action of $K_\C$ corresponds to changes of trivialization. Though we can not yet prove this quantization result, we show here that factorization homology can be expressed in terms of $DK_q$-C$^*$-algebras (Theorem \ref{thmB}). 

In order to compute factorization homology, one has to understand $\cZ \Hilb(\cC)$-module C$^*$-categories, a priori a difficult problem due to the lack of both rigidity, semisimplicity and finiteness. The main technical result of this paper (Theorem \ref{themA}) allows one to pass from module C$^*$-categories over $\cZ\Hilb(\cC)$ to bimodule C$^*$-categories over $\cC$, enabling the application of the results in \cite{hataishi2025categorical}. In the case $\cC = \fdRep(K_q)$ this is how one obtains the $DK_q$-C$^*$-algebras, and for generic $\cC$ it allows to express factorization homology in terms of extensions of Popa's symmetric enveloping algebras (Theorem \ref{thmC}). We expect that these extensions describe the algebras of observables in the 2-dimensional TQFT associated with factorization homology.

\medskip
Lets us now describe in more detail the content of the paper. In passing from $\cC$ to $\cZ\Hilb(\cC)$, semi-simplicity, and therefore rigidity, is lost, and the result is an immense C$^*$-category. This is well illustrated in the case of unitary representations of compact quantum groups. It is true that the unitary Drinfeld center is the category of unitary representations of the Drinfled double, which is now only a {\em locally} compact quantum group. Even classically, it is often very difficult to understand the category of unitary representations of a generic locally compact group.

 In the fusion case, Müger's works \cites{MR1966524,MR1966525} assert that the Drinfeld center of a fusion category 
$\cC$ is related to $\cC^{mp} \boxtimes \cC$ through a Morita context, where $\cC^{mp}$ is the category $\cC$ with opposition monoidal structure, and $\boxtimes$ is the tensor product of linear categories. Translated to the monadic language of \cites{MR2355605,MR2793022,MR2869176,MR3079759}, this tantamounts to the existence of an algebra object $\cS$ in $\cC^{mp} \boxtimes \cC$ such that
\begin{itemize}
	\item the category $\RMod(\cS)$ of right $\cS$-modules is equivalent to $\cC$,
	\item the category $\LMod(\cS)$ of left $\cS$-modules is equivalent to $\cC^{mp}$,
	\item the category $\Bim(\cS)$ is $\cS$-bimodules is equivalent to the Drinfeld center $\cZ\cC$,
	\item $\RMod(\cS) \underset{\cC^{mp} \boxtimes \cC}{\boxtimes} \LMod(\cS) \simeq \cC  \underset{\cC^{mp} \boxtimes \cC}{\boxtimes} \cC^{mp} \simeq \cZ(\cC)$.
\end{itemize}
A consequence of this is that the categories of modules for the Drinfeld center is equivalent to the category of modules for $\cC^{mp} \boxtimes \cC$. Also, there is a canonical choice for $\cS$ up to isomorphism: if $\Irr(\cC)$ is a choice of representatives for the isomorphism classes of simple objects in $\cC$, $\cS$ can be taken to be
\[
\cS = \bigoplus_{a \in \Irr(\cC)} \bar{a} \boxtimes a \ ,
\]
where $\bar{a}$ is the dual of $a$.

The basic idea of this paper is to investigate to which extent the above picture makes sense in the unitary and non-finite context. The inifinite version of the canonical algebra $\cS$ has already appeared both in quantum groups, as the implementation of the quantum-duality principle (\cites{MR1015339,MR1302020}), and in subfactors as part of the standard invariant of the symmetric enveloping inclusions (\cites{MR3509018,MR3687214}). Given a UTC $\cC$, as explained in \cite{MR3687214}, $\cS$ is a W$^*$, or von Neumann, algebra object in the algebraic ind-completion of $\cC^{mp} \boxtimes \cC$. The following is the main result in this paper.

\begin{theoremAlph}[Corollary \ref{cor:equivalencehalfbraidingsunitarybimodulesoverSEalgebra}]
	The unitary Drinfeld center $\cZ \Hilb(\cC)$ of a unitary tensor category is equivalent to the category  $\Bim_{\Hilb(\cC^{mp} \boxtimes \cC)}(\cS)$ of unitary bimodules for the canonical W$^*$-algebra object $\cS$.
	\label{themA}
\end{theoremAlph}

Unitarily braided C$^*$-tensor categories can be interpreted topologically. Consider the category $\Disk_2$ which has disjoint unions $D^{\sqcup k}$ of two dimensional disks as objects, and smooth embeddings $D^{\sqcup k} \to D^{\sqcup l}$ as morphisms. Disjoint union gives to $\Disk_2$ the structure of a symmetric monoidal category. Considering isotopies between embeddings as 2-morphisms, $\Disk_2$ becomes a symmetric monoidal (2,1)-category, i.e., a 2-category for which all 2-morphisms are invertible. Then a unitarily braided C$^*$-tensor category $\cZ$ determines a symmetric monoidal functor $\Disk_2 \to \CCat$, where $\CCat$ is the symmetric monoidal (2,1)-category of C$^*$-tensor categories, $*$-functors and unitary natural isomorphisms, equipped with the maximal tensor product $\boxtimes$ of C$^*$-categories. The category $\cZ$ is recovered from the functor as the value of the disk.

Thanks to \cite{MR3007088} and  \cite{MR3431668}, it follows that the functor $\Disk_2 \to \CCat$ defined by a  C$^*$-tensor category $\cZ$ with a unitary braiding extends to a symmetric monoidal functor
\[
\int_{(-)} \cZ : \Surf \to \CCat 
\]
from the category of compact surfaces, seen as a symmetric monoidal (2,1)-category in similar fashion as $\Disk_2$.  This functor is called {\em factorization homology}  with coefficients in $\cZ$ (\cites{MR3431668}), and satisfies an excision property in terms of collar-gluing of surfaces.  It naturally leads to the study of pointed module C$^*$-categories for unitary Drinfeld centers (a pointed module category is a module category with a distinguished object). 

We use Theorem \ref{themA} to study $\int_{(-)} \cZ \Hilb(\cC)$, factorization homology with coefficients in the unitary Drinfeld center of a UTC $\cC$. Using Theorem \ref{themA}, the following can be proven.

\begin{theoremAlph}[Proposition \ref{prop:frommodulerforthecentertocentrallypointedbimodules}]
	There exists a forgetful functor from 
	\[
	\lbrace \text{pointed $\cZ \Hilb(\cC)$-bimodule C$^*$-tensor categories} \rbrace \to \lbrace \text{ centrally pointed $\cC$-bimodule C$^*$-categories} \rbrace \ ,
	\]
	given at the level of objects by 
	\[
	(\cM,m) \mapsto \left( \cM \underset{\cZ \Hilb(\cC)}{\boxtimes} \Hilb(\cC), m \boxtimes \un \right) \ . 
	\]
	\label{thmB}
\end{theoremAlph} 

A pointed bimodule C$^*$-category $(\cM,m)$ for $\cC$ is called {\em centrally pointed} if on the distinguished object $m$ the left and right $\cC$ actions are unitarily equivalent. It was show in \cite{hataishi2025categorical} that, when $\cC$ is the category $\Rep \bK$ of unitary representations of a compact quantum group, such bimodules correspond to continuous action of the Drinfeld double $D \bK$ on unital C$^*$-algebras, via a generalization of the Tannaka-Krein duality for quantum group actions. In particular, factorization homology for unitary Drinfeld center gives rise to continuous $D \bK$-C$^*$-algebras. More precisely, we have the following.

\begin{theoremAlph}[Corollary \ref{cor:DrinfelddoubleactionsfromFH}]
	Let $\Sigma$ be a framed surface with non empty boundary. Given a compact quantum group $\bK$, there is a unital C$^*$-algebra $B_\Sigma$, a continuous actions of $D\bK^{\times n}$ on $B_\Sigma$, where $n$ is the number of boundary components of $\Sigma$, such that
	\[
	\int_{\Sigma} \Rep(D\bK) \underset{\Rep(D\bK^{\times n})}{\boxtimes} \Rep(\bK^{\times n}) \simeq \Hilb^{\bK^{\times n}}_{B_\Sigma} \ , 
	\]
	as a $\Rep(\bK^{\times n})$-bimodule C$^*$-tensor category, where $\Hilb^{\bK^{\times n}}_{B_\Sigma}$ is the category of $\bK^{\times n}$-equivariant Hilbert C$^*$-modules over $B_\Sigma$. The correspondence $\Sigma \mapsto B_\Sigma$ is functorial with respect to boundary preserving framed embeddings.
	\label{thmC}
\end{theoremAlph}

In future works, factorization homology with coefficients in $\Rep (D\bK)$ will be interpreted as topological quantum gauge theory.
\medskip

The $D\bK$-C$^*$-algebras of Theorem \ref{thmC} are obtained from the fiber $\Rep \bK \to \Hilb$ of the unitary representation category of $\bK$. There are UTCs that do not admit fiber functor, and therefore are not equivalent to the unitary representation category of a compact quantum group, e.g. Fibonacci, Ising and Haagerup unitary fusion categories. But, given a UTC $\cC$, there is always a unital C$^*$-algebra $A$ with trivial center and a fully-faithful unitary tensor functor $F:\cC \to \Corr(A)$ to the category of Hilbert C$^*$-correspondences over $A$ (\cites{MR4139893,kitamura2024actions}). 

Building on the previous works (\cites{MR3948170,palomares2023discrete}), a generalization of the Tannaka-Krein duality for quantum group actions is proven in \cite{Hataishi_Palomares_2025}, functorially associating to a fully faithful unitary tensor functor $F:\cC \to \Corr(A)$ and a pointed module C$^*$-category $(\cM,m)$ for $\cC$ an extension $A \subset B_{(\cM,m)}$ of C$^*$-algebras. A pointed module C$^*$-category over $\cC$ is the same as a C$^*$-algebra object in the algebraic ind-completion of $\cC$ (\cite{MR3687214}), and thus the generalized Tannaka-Krein duality theorem is a correspondence associating concrete operator algebraic extensions to C$^*$-algebra objects.

How should the generalized Tannaka-Krein duality of \cite{Hataishi_Palomares_2025} work for centrally pointed module C$^*$-categories, and how can an analogue of Theorem \ref{thmC} be stated when, instead of a fiber functor $\Rep \bK \to \Hilb$, a fully faithful functor $F:\cC \to \Corr(A)$ is given? The answer involves the symmetric enveloping inclusion. Indeed, the functor $F$ induces a fully faithful unitary tensor functor $ \cC^{mp} \boxtimes \cC \to \Corr(A^{\op} \otimes A)$. Taking the canonical algebra $\cS$ as a C$^*$-algebra object in the algebraic ind-completion of $\cC^{mp} \boxtimes \cC$, there is an associated extension $A^{\op} \otimes A \subset B_F$, and $B_F$ is called the {\em symmetric enveloping algebra} associated to $F$. The following is then proved.

\begin{theoremAlph}[Theorem \ref{thm:extensionsofsymmetricenvelopingalgebra}]
	Suppose $\cC$ is a unitary tensor category and $F: \cC \to \Corr(A)$ is a fully faithful unitary tensor functor. Then there exists a functor 
	\[(\cM,m) \mapsto B_{(\cM,m)} \supset B_F \]
	from the category of pointed module C$^*$-categories over the unitary Drinfeld center of $\cC$ to the category of extensions of the symmetric enveloping algebra $B_F$ associated to $F$.
	\label{thmD}
\end{theoremAlph}

Applying Theorem \ref{thmD} to factorization homology gives the following result.

\begin{theoremAlph}[Corollary \ref{cor:extensionofSEfromFH}]
	Let $\cC$ be a unitary tensor category and $F: \cC \to \Corr(A)$ a fully faithful unitary tensor functor. Given a framed surface $\Sigma$ with non empty boundary, there is an associated extension  $B_F^{\otimes n} \subset B_\Sigma$ of C$^*$-algebras, where $n$ is the number of boundary components of $\Sigma$, and a fully-faithful unitary functor
	\[
	\int_\Sigma \cZ \Hilb(\cC) \underset{\cZ \Hilb(\cC)^{\boxtimes n}}{\boxtimes} \Hilb(\cC) \to \Corr^l(B_\Sigma,B_F^{\otimes n}) \ ,
	\]
	of $\cC^{\boxtimes n}$-bimodule C$^*$-categories, where the target is the category of left Hilbert $(B_\Sigma,B_F^{\otimes n})$-C$^*$-correspondences. This construction is functorial in $\Sigma$. Moreover, the modular group $\Gamma(\Sigma)$ of $\Sigma$ relative to the boundary acts on $B_\Sigma$ by $*$-automorphims, and $B_F^{\otimes n}$ is globally fixed under this action. 
	\label{thmE}. 
\end{theoremAlph}

\subsection{Acknowledgements}
I would like to thank Makoto Yamashita, Sergey Neshveyev, Corey Jones, David Jordan, Christian Voigt and Thomas Wasserman for helpful conversations on topics related to this paper.  L.H. is supported by the Engineering and Physical Sciences Research Council (EP/X026647/1). For the purpose of Open Access, the author has applied a CC BY public copyright licence to any Author Accepted Manuscript (AAM) version arising from this submission.

\section{Preliminaries}

\subsection{Notation}

The space of linear maps between two vector spaces $V_1$ and $V_2$ is denoted by $\cL(V_1,V_2)$. The space of adjointable operators between Hilbert spaces $H_1$ and $H_2$ will be denoted by $\bB(H_1,H_2)$.

\begin{remark}
	For notational simplicity, when $a$ and $b$ are objects of a strict tensor category, we may sometimes write $ab$ for their tensor product $a \otimes b$.
\end{remark}

\subsection{Unitary Tensor Categories}
In this preliminary section the basics of the theory of unitary tensor categories is spelled, mostly in order to fix the notation that is used througout the paper. For an in-depth study, the reader is referred to \cite{MR3687214,MR3204665}.  As a motivation, keep in mind that unitary tensor category is the axiomatization of the structure of the category of finite dimensional unitary representations of a compact quantum group. 

\begin{definition}
	A C$^*$-category $\cC$ consists of
	\begin{itemize}
		\item  A category $\cC$ enriched over the category of Banach spaces and linear contractions;
		\item An involutive functor $* : \cC \to \cC^{\op}$ which on the objects act as the identity. In other words, it is a natural collection of anti-linear maps $\cC(a,b) \to \cC(b,a)$ which squares to the identity,
	\end{itemize}
	such that
	\begin{itemize}
		\item $\cC$ admits direct sums;
		\item For each $a \in \cC$, the endomorphism algebra $\cC(a)$, with the given Banach space structure and the involution $*: \cC(a) \to \cC(a)$, is a C$^*$-algebra.
	\end{itemize}
	\label{def:C*category}
\end{definition}

\begin{remark}
	The existence of direct sums in the defintion of C$^*$-categories is a simplifying assumption.
\end{remark}

If $\cC$ and $\cD$ are C$^*$-categories, a functor $F: \cC \to \cD$ is said to be {\em unitary} if $F \circ * = * \circ F$. Because endomorphism algebras are C$^*$-algebras, this automatically implies continuity of $F$ on morphism spaces with respect to the given Banach space structures.

\begin{definition}
	Let $\cC$ be a C$^*$-category. A tensor structure $\otimes:
	\cC \times \cC \to \cC$ is a {\em $*$-tensor structure} if, given $a,a',b,b' \in \cC$, morphisms $f: a \to a'$ and $g: b \to b'$, 
	\[
	(f \otimes g)^* = f^* \otimes g^* : a' \otimes b' \to a \otimes b \ .
	\]
	A C$^*$-category equipped with a $*$-tensor structure is called a {\em C$^*$-tensor category}. A rigid C$^*$-tensor category is a C$^*$-tensor category which is rigid as a tensor category.
	\label{def:C*tensorstructure}
\end{definition}

The following notation is used for the duality functor in rigid tensor categories. On objects, we write $a \mapsto \bar{a}$. On morphisms, we write $(f:a \to b) \mapsto (f^\vee: \bar{b} \to \bar{a})$. Unit objects in tensor categories will generically be denoted by $\un$.

\begin{definition}
	An object $a$ in C$^*$-category $\cC$ is {\em simple} if $\End(a) = \C \cdot \id_a$. A rigid C$^*$-tensor category is said to be a {\em unitary tensor category} of its unit $\un$ is simple.
	\label{def:unitarytensorcategories}
\end{definition}

Unitary tensor categories are automatically semisimple and spherical. To simplify the notation, we work only with strictified unitary tensor categories: the tensor structure is strict and the double dual functor is the identity, $(\bar{\bar{(-)}}, (-)^{\vee \vee}) = \id$. If  a given unitary tensor category is not strict on the nose, we replace it by its strictification without mentioning.

Recall that a functor between linear categories with direct sums is additive if it is linear at the level of both objects and morphisms.

\begin{definition}
	Let $\cC$ be a unitary tensor category. The {\em algebraic ind-completion} $\Vect(\cC)$ of $\cC$ is the category of additive functors $\cC^{\op} \to \Vect$. Morphisms in this category are linear natural transformations.
\end{definition}

To specify an object $V \in \Vect(\cC)$ up to isomorphism, it suffices to specify its values, or {\em fibers}, at the objects in $\Irr(\cC)$, i.e., a family $\{V(a)\}_{a \in \Irr(\cC)}$ of vector spaces indexed by $\Irr(\cC)$. Another useful perspective is that of understanding the object $V = \{V_a\}_{a \in \Irr(\cC)} \in \Vect(\cC)$ as a  formal direct sum
\[ V = \bigoplus_{a \in \Irr(\cC)} X \odot V(a) \]
of objects in $\cC$, where the irreducible object $a$ appears $\dim V(a)$ times as a direct summand. We say that $V(a)$ is the {\em multiplicity space} of $a$ in $V$. Moreover, given $V_1 = \{V_1(a)\}_a$ and $V_2=\{V_2(a)\}_a$ in $\Vect(\cC)$, it follows that
\[ \Vect(\cC)(V_1,V_2) \simeq \prod_{a \in \Irr(\cC)} \cL(V_1(a),V_2(a)) , \]
where $\cL(\cdot,\cdot)$ denotes the space of linear maps.

The category $\cC$ embeds into $\Vect(\cC)$ as a linear full subcategory, and the tensor structure $\otimes$ in $\cC$ can be extended to a tensor structure $\odot$ on $\Vect(\cC)$: for $V_1$ and $V_2$ as above,
\[ (V_1 \odot V_2)(a) : = \bigoplus_{b,c \in \Irr(\cC)} \cC(a, b \otimes c) \odot V_1(b) \odot V_2(c) . \]

\begin{definition}
	The {\em unitary ind-completion} of a unitary tensor category $\cC$ is the category $\Hilb(\cC)$ with
	\begin{itemize}
		\item objects: additive functors $\cC^{\op} \to \Hilb$;
		\item morphisms: uniformly bounded linear natural transformations. That is, given $H_1, H_2 \in \Hilb(\cC)$, a morphism $\eta \in \Hilb(\cC)(H_1,H_2)$ is a natural family $\{ \eta_a \in \bB(H_1(a),H_2(a))\}_{a \in \cC}$ such that
		\[ \underset{a \in \cC}{\sup} \| \eta_a \| < \infty . \]
	\end{itemize}
\end{definition}

An object $H \in \Hilb(\cC)$ is determined by a family $\{H(a)\}_{a \in \Irr(\cC)}$ of Hilbert spaces. We have
\[ \Hilb(\cC)(H_1,H_2) \simeq \prod_{a \in \Irr(\cC)}^{\ell^\infty} \bB(H_1(a),H_2(a)) . \]

The category $\Hilb(\cC)$ is a C$^*$-category (actually a W$^*$-category), and $\cC$ is a full C$^*$-subcategory of it. The tensor structure of $\cC$ extends to $\Hilb(\cC)$; seeing $\cC(a,b  \otimes c)$ as a Hilbert space,
\[ (H_1 \otimes H_2)(a) = \bigoplus_{b,c \in \Irr(\cC)}^{\ell^2} \cC(a, b \otimes c) \otimes H_1(b) \otimes H_2(c) , \]
where on the $(b,c)$-direct summand the inner-product is scaled by a factor of $(d_b d_c)^{-1}$. The purpose of this renormalization is to have a clear interpretation of the induced inner product by means of graphical calculus (see \cite{MR3687214}, Section 2).

If $a \in \Irr(\cC)$, for any $H \in \Hilb(\cC)$ the morphism space $\Hilb(\cC)(a,H)$ is a Hilbert space with inner product $\langle f, g \rangle = f^* \circ g$. We fix a choice of an orthonormal basis for it, and denote it by $O(a,H)$.

\medskip

\subsection{C$^*$-algebra objects} 
Since $\Vect(\cC)$ is a tensor category, we can consider algebra objects in it. An algebra object  in $\Vect(\cC)$ consists of a functor $\bA: \cC^{\op} \to \Vect$ together with a lax-natural transformation 
\[\mu^{\bA} = \{ \mu^{\bA}_{a,b} : \bA(a) \odot \bA(b) \to \bA(a \otimes b) \ | \ a,b \in \cC\} ,\]
satisfying coherence with respect to the pentagon diagram associated to the orderings of triple tensor products. 

\begin{definition}
	A $*$-structure on an algebra object $\bA$ is a conjugate linear natural transformation 
	\[  j^{\bA} = \{ j^{\bA}_a: \bA(a) \to \bA(\bar{a})\}_{a \in \cC} , \]
	satisfying
	\begin{itemize}
		\item involution: $ j^{\bA}_{\bar{a}} \circ j^{\bA}_a = \id_{\bA(a)}$, under the identification $\bar{\bar{a}} = a$;
		\item unitality: $ j^{\bA}_{\un} = \id_{\bA(\un)}$, under the identification $\bar{\un} \simeq \un$;
		\item monoidality: the diagram
		\begin{center}
			\begin{tikzcd}
				\bA(a) \odot \bA(b) \arrow[rr, " j^{\bA}_a \odot  j^{\bA}_b"] \arrow[dd, "\mu^{\bA}"'] &  & \bA(\bar{a}) \odot \bA(\bar{b}) \arrow[dd, "\mu^{\bA}"]              \\
				&  &                                                                     \\
				\bA(a \otimes b) \arrow[rr, " j^{\bA}_{a \otimes b}"']                         &  & \bA(\bar{a}\otimes \bar{b}) \simeq \bA(\overline{b \otimes a})
			\end{tikzcd}
		\end{center}
		commutes.
	\end{itemize}
\end{definition}

If $\bA = (\bA, \mu^{\bA})$ is an algebra object in $\Vect(\cC)$, then the vector spaces $\bA(\bar{a} \otimes a)$ become algebras, with multiplication given by
\[ \bA(\bar{a} \otimes a) \odot \bA(\bar{a} \otimes a) \overset{\mu^{\bA}}{\to} \bA( \bar{a} \otimes a \otimes \bar{a} \otimes a) \overset{\bA(\id_{\bar{a}} {\otimes \bar{R}_a^* \otimes \id_a)}}{\to} \bA(\bar{a} \otimes a) . \]
If $j^\bA$ is a $*$-structure on $\bA$, then $\bA(\bar{a} \otimes a)$ becomes $*$-algebra, with involution given by $ j^{\bA}_{\bar{a}\otimes a}$.

\begin{definition}
	Let $\bA = (\bA,\mu^{\bA},  j^{\bA})$ be a $*$-algebra in $\Vect(\cC)$. We say that $\bA$ is a C$^*$-algebra object if its fibers are naturally equipped with Banach space structures and, for each object $a$, the $*$-algebra $\bA(\bar{a} \otimes a)$ is a C$^*$-algebra with respect to the given Banach norm.
\end{definition}

\begin{remark}
	C*-algebra objects can be equivalently defined in terms of C*-module categories as follows: 
	Given a $*$-algebra object $\bA$ as above, the category $\cM_{\bA}$ generated under finite direct sums and idempotent completion by objects of the form $\bA \odot a$, with $a \in \cC$, and with morphisms 
	\[ \cM_{\bA}(\bA \odot a, \bA \odot b) := \Vect(\cC)(a, \bA \odot b) \simeq \bA(a \otimes \bar{b}) \]
	has a right $\cC$-module structure. $\bA$ is a C$^*$-algebra object if and only if $\cM_{\bA}$ is a $\cC$-module C$^*$-category \cite[Definition 25]{MR3687214}. 
\end{remark}

To finish this preliminary section, we set up the notation for Frobenius reciprocity maps in UTCs.

\begin{definition}
    If $\cC$ is a unitary tensor category, $U,U_1,U_2 \in \cC$, 
    \[ \Phi^l_U : \cC(U \otimes U_1,U_2) \to \cC(U_1,\bar{U} \otimes U_2) \]
    denotes the left Frobenius reciprocity map, and 
    \[
    \Phi_U^r: \cC(U_1 \otimes U,U_2) \to \cC(U_1, U_2 \otimes \bar{U})
    \]
    denotes the right Frobenius reciprocity map. The composition
    \[
    \cC(U_1, U_2) \overset{\Phi_{U_1}^l}{\longrightarrow} \cC(\un_{\cD},\bar{U}_1 \otimes U_2) \overset{(\Phi_{\bar{U}_2}^r)^{-1}}{\longrightarrow} \cC(\bar{U}_2 , \bar{U}_1)  
     \]
    is denoted by $(-)^\vee: f \mapsto f^\vee$, and can equivalently be characterized by the equality
    \[
    (\id_{\bar{U}_1} \otimes f) \circ R_{U_1} = (f^\vee \otimes \id_{U_2}) \circ R_{U_2} . 
    \]
\end{definition}

\subsection{Drinfeld center}


\section{Monadicity of the center}

This section enhances the results of  \cites{MR2355605,MR2793022,MR2869176,MR3079759}, summarized in \cite{MR3674995},  in the direction of giving a monadic characterization of $\cZ \Vect(\cC)$ when $\cC$ is a unitary tensor category. The strategy and even the techniques used here are recicled from the given references, the difference being that at some moments extra care is needed when dealing with ind-objects.

\subsection{Centralizer endofunctors}
Given $X \in \Vect(\cC)$, consider the functor $C_X: \cC^{\op} \times \cC \to \Vect(\cC)$ given by
$$C_X(U^{\op},V) := \bar{U} \odot X \odot V \ . $$
If $(f,g):(U,V) \to (U', V')$ is a morphism in $\cC^{\op} \times \cC$, then $C_X(f,g) = f^\vee \odot \id_X \odot g$. 

\begin{lemma}\label{lemma:shiftedind-objects}
	Let $X \in \Vect(\cC)$  and $U \in \cC$. Then there is a canonical natural isomorphism
	\[ 
	X(U \otimes  (-)) \simeq \bar{U} \odot X 
	\]
	in $\Vect(\cC)$.
\end{lemma}
\begin{proof}
	Indeed, for any $a \in \cC$,
	\begin{align*}
		X(U \otimes a) \simeq \Vect(\cC)(U \otimes a,X) \simeq \Vect(\cC) (a, \bar{U} \odot X) \simeq (\bar{U}  \odot X)(a) \ ,
	\end{align*}
	where the penultim isomorphism is given by the extension $\Phi^l_U$ to $\Vect(\cC)$, which is well defined.
\end{proof}

\begin{definition}
	Define an endofunctor $Z_{\Vect} \in \End(\Vect(\cC))$ as follows. For $X \in \Vect(\cC)$, $Z_{\Vect}(X)$ is the object in $\Vect(\cC)$ defined by
	\[
	Z_{\Vect}(X) = \bigoplus_{a \in \Irr(\cC)} \bar{a} \odot X \odot a \ .
	\]
	If $f \in \Vect(\cC)(X,Y)$, then $Z_{\Vect}(f) = \bigoplus_{a \in \Irr(\cC)} \id_{\bar{a}} \odot f \odot \id_{{a}}$. We call $Z_{\Vect}$ the centralizer endofunctor of $\Vect(\cC)$.
	Lemma \ref{lemma:shiftedind-objects} gives an identification, 
	\[
	\cC^{\op} \ni U \mapsto Z_{\Vect}(X)(U) \simeq \bigoplus_{a \in \Irr(\cC)} X(a \odot U \odot \bar{a}) .
	\]
	\label{def:algebraiccentralizer}
\end{definition}

For the next definition, observe that if $(X_i)_{i \in I}$ is an arbitrary family of objects in $\Hilb(\cC)$, we can definite their Hilbert space object direct sum $\bigoplus_{i \in I}^{\ell^2} X_i$ as the additive functor $\cC^{\op} \to \Hilb$ given by
\[
\cC^{\op} \ni U \mapsto \bigoplus_{i \in I}^{\ell^2} X_i(U) \ .
\]

\begin{definition}
	For $X \in \Hilb(\cC)$, let $Z_{\Hilb}(X) \in \Hilb(\cC)$ be 
	\[
	Z_{\Hilb}(X) = \bigoplus_{a \in \Irr(\cC)}^{\ell^2} \bar{a} \otimes X \otimes a \ ,
	\]
	
	Given $f \in \Hilb(\cC)(X,Y)$, let $Z_{\Hilb}(f) = \bigoplus_{a \in \Irr(\cC)} \id_{\bar{a}} \otimes f \otimes \id_{a}$. Then $Z_{\Hilb} \in \End(\Hilb(\cC))$ is a unitary functor, and we call it the centralizer functor of $\Hilb(\cC)$. Lemma~\ref{lemma:shiftedind-objects} gives an identification
	\[
	\cC^{\op} \ni U \mapsto Z_{\Hilb}(X)(U) \simeq \bigoplus_{a \in \Irr(\cC)}^{\ell^2} X(a\otimes U \otimes \bar{a})  . 
	\]
	\label{def:unitarycentralizer}
\end{definition}

\begin{definition}
	Let $F: \cC^{\op} \times \cC \to \Hilb(\cC)$ be a bilinear unitary functor, and let $K \in \Hilb(\cC)$. A {\em dinatural transformation $\theta: F \to K$ in $\Hilb(\cC)$} is a family 
	\[
	\theta = \{ \theta_U \in \Hilb(\cC)(F(U^{\op},U) , K)\}_{U \in \cC}
	\]
	such that 
	\begin{itemize}
		\item for every $f \in \cC(U_1,U_2)$, the following diagram commutes:
		\begin{center}
			\begin{tikzcd}
				& {F(U_2^{\op},U_1)} \arrow[rd, "{F(\id_{U_2^{\op}},f)}"] \arrow[ld, "{F(f^{\op},\id_{U_1})}"'] &                                             \\
				{F(U_1^{\op},U_1)} \arrow[rd, "\theta_{U_1}"'] &                                                                                         & {F(U_2^{\op},U_2)} \arrow[ld, "\theta_{U_2}"] \\
				& K                                                                                       &                                            
			\end{tikzcd}
		\end{center}
		
		\item $\sup_{U \in \cC} \| \theta_U\| < \infty$.
	\end{itemize}
	
	\label{def:dinaturaltransfinHilb}
\end{definition}

\begin{lemma}
	Let $X \in \Vect(\cC)$. For each $U \in \cC$ and $a \in \Irr(\cC)$,  
	\[
	\rho_{X,U} :=  \sum_{a \in \Irr(\cC)} \sum_{\omega \in O(a,U)} ((\omega^\vee)^* \odot \id_X \odot \omega^*)
	\]
	defines a morphism $\rho_{X,U} \in \Vect(\cC)(\bar{U} \odot X \odot U, Z_{\Vect}(X))$. Moreover, $\rho_X:= \{\rho_{X,U}\}_{U \in \cC}$ is a dinatural transformation $C_X \to Z_{\Vect}(X)$ in $\Vect(\cC)$. If $X \in \Hilb(\cC)$, the analogous  formula with the algebraic tensor product $\odot$ replaced by the Hilbert space tensor product $\otimes$ defines a morphism in $ \rho_{X,U} \in \Hilb(\cC)(\bar{U} \otimes X \otimes U, Z_{\Hilb}(X))$. In this case, $\rho_X$ is a dinatural transformation $C_X \to Z_{\Hilb(X)}$ in $\Hilb(\cC)$.
	\label{lemma:centraldinaturaltransformation}
\end{lemma}

\begin{proof}
	We prove the Lemma only in the unitary case, the algebraic version being simpler.  Given $U \in \cC$, there are finitely many $a \in \Irr(\cC)$ for which $\Hom_\cC(a,U) \neq 0$. Thus the sum appearing in the definition of $\rho_{X,U}$ can be written as a sum over a finite index set, and it is well defined as a morphism in $\Hilb(\cC)$ due to the tensor-structure that this category canonically inherits from $\cC$. 
	
	Let us show that $\rho_X$ satisfies the dinaturality condition. Let $f \in \cC(U_1,U_2)$. Then
	\begin{align*}
		\rho_{X,U_1} \circ C_X(f^{\op},\id_{U_1}) & = \sum_{a \in \Irr(\cC)} \sum_{w \in O(a,U_1)} ((\omega^\vee)^* \otimes \id_X \otimes \omega^*) (f^\vee \otimes \id_X \otimes \id_{U_1}) \\
		& = \sum_{a \in \Irr(\cC)} \sum_{w \in O(a,U_1)} ((f\circ \omega^*)^\vee \otimes \id_X \otimes \omega^*) 
		\\
		& = \sum_{a \in \Irr(\cC)} \sum_{v \in O(a,U_2)} (( v^*)^\vee \otimes \id_X \otimes (f \circ v^*))  \\
		& = (\id_{\bar{U}_2} \otimes \id_X \otimes f) \circ \rho_{X,U_2}
	\end{align*}
	
	To finish the proof, we have to verify that $(\rho_{X,U})_{U \in \cC}$ is uniformly bounded. For this, observe that 
	\begin{align*}
		\rho_{X,U}^* \rho_{X,U} & = \sum_{i,j \in \Irr(\cC)} \sum_{\omega_i \in O(U_i,U)} \sum_{\omega_j \in O(U_j,U)} \omega_j^{\vee *} \omega_i^\vee \otimes \id_X \otimes \omega_j \omega_i^*\\
		&= \sum_{i \in \Irr(\cC)} \sum_{\omega \in O(U_i,U)} (\omega \omega^*)^\vee \otimes \id_X \otimes \omega \omega^* \\
		& \leq \id_{\bar{U}} \otimes \id_X \otimes \id_U \ . 
	\end{align*}
	Consequently,
	$$\| \rho_{X,U}\|^2 = \|\rho_{X,U}^* \rho_{X,U} \| \leq 1 \ . $$
\end{proof}

\begin{theorem}
	Let $X \in \Vect(\cC)$. Then the centralizer $Z_{\Vect}(X)$ is the coend of $C_X$ in $\Vect(\cC)$. Similarly, given $X \in \Hilb(\cC)$, $Z_{\Hilb}(X)$ is the coend of $C_X$ in $\Hilb(\cC)$.
	\label{thm:definitionmonadZ}
\end{theorem}

\begin{proof}
	Once again, the proof of the algebraic version of the claim is similar to the unitary case, but without the analytical issues and therefore simpler. We shall only prove the unitary version then.
	
	Suppose $\theta = \{ \theta_U: \bar{U} \otimes X \otimes U \to K\}_{U \in \cC}$ is a dinatural transformation $C_X \to Z_{\Hilb(\cC)}$ in $\Hilb(\cC)$. Recall that in particular $\| \theta \| = \sup_U \|\theta_U \| < \infty$. From $\theta$ we want to produce a morphism $\Theta: Z_\Hilb(X) \to K$, satisfying $\Theta \circ \rho = \theta$, and then show that $\Theta$ is the unique morphism having this property.
	
	Each $\theta_a$, for $a \in \Irr(\cC)$, is a uniformly bounded collection of bounded linear maps 
	\[\{ \theta_{a,b}: (\bar{a} \otimes X \otimes  a)(b) \simeq X(a \otimes b \otimes \bar{a}) \to K(b) \ \}_{b \in \Irr(\cC)}.
	\]
	
	Given 
	\[\xi = (\xi_a)_{a \in \Irr(\cC)} \in Z_{\Hilb}(X)(b) = \bigoplus_{a \in \Irr(\cC)}^{\ell^2} X(a \otimes b \otimes \bar{a}),
	\]
	the net $\{ \sum_{a \in I} \theta_{a,b}(\xi_a) \ | I \Subset \Irr(\cC)\}$, where $\Subset$ means finite contained, converges in $K(b)$ due to the uniform boundedness of $\theta$. Indeed, given $\epsilon > 0$ 
	\begin{align*}
		\| \sum_{a \in \Irr(\cC) \setminus I} \theta_{a,b}(\xi_a) \|^2 = \sum_{a,c \in \Irr(\cC) \setminus I} \langle \theta_{c,b}^* \theta_{a,b} (\xi_a),\xi_c \rangle \leq \|\theta\|^2 \sum_{a \in \Irr(\cC) \setminus I} \|\xi_a\|^2 < \epsilon
	\end{align*}
	for every $I \Subset \Irr(\cC)$ sufficiently large. We have then a well defined morphism $\Theta: Z_{\Hilb}(X) \to K$ associating to $(\xi_a)_{a \in \Irr(\cC)} \in Z_{\Hilb}(X)(b)$ as above the vector $\sum_a\theta_{a,b}(\xi_a) \in K(b)$. Naturality of $\theta$ implies then
	$$\Theta \circ \rho_{X,U} = \theta_U \ \forall \ U \in \cC \ . $$
	
	More concisely, we can write
	$$\Theta = \sum_{a \in \Irr(\cC)} \theta_a  \ . $$
	Now we show uniqueness. If $\tilde{\Theta}: Z_{\Hilb}(X) \to K$ is another morphism in $\Hilb(\cC)$ such that $\tilde{\Theta} \circ  \rho_{X,U} = \theta_U$ for each $U \in \cC$, it holds in particular
	$$\theta_a = \tilde{\Theta} \circ \rho_{X,a} \ . $$
	But $\rho_{X,a}$ is just the inclusion $$ \bar{a} \otimes X \otimes a \hookrightarrow \bigoplus_{c \in \Irr(\cC)} \bar{c} \otimes X \otimes c = Z_{\Hilb}(X) \ . $$
	Thus, it must hold
	$$\tilde{\Theta} = \sum_{a \in \Irr(\cC)} \theta_a = \Theta \ . $$
\end{proof}

\subsection{Free half-braidings}
We shall now discuss how to enhance the functor $Z_{\Hilb} \in \End(\Hilb(\cC))$ to a functor $\Hilb(\cC) \to \cZ \Hilb(\cC)$. We start by introducing the {\em free half-braidings}.

Let $X \in \Hilb(\cC)$. For each $U \in \cC$, let $(\sigma^X_U)_{a,b}: \bar{a} \otimes X \otimes a \otimes U \to U \otimes\bar{b} \otimes X \otimes b$ be given by
\[
(\sigma^X_U)_{a,b} : = \left( \frac{d_a}{d_b} \right)^{1/2} \sum_{\omega \in O(b,a \otimes U)} (\id_U \otimes \omega^\vee \otimes \id_X\otimes \omega^*) (\bar{R}_U \otimes \id_{\bar{a}} \otimes \id_X \otimes \id_{a}).
\]

\begin{theorem}[\cite{MR3509018}, Theorem 3.4]
	Given $X \in \Hilb(\cC)$, $\sigma^X := \{ (\sigma^X)_{a,b}\}_{a,b \in \Irr(\cC)}$ is a unitary half-braiding on $Z_{\Hilb}(X)$. 
	\label{thm:freehalfbraiding}
\end{theorem}

\begin{definition}
	The functor $\cF_{\Hilb}: \Hilb(\cC) \to \cZ\Hilb(\cC)$ given by
	\begin{center}
		\begin{tikzcd}
			X \arrow[dd, "f"'] \arrow[rr, "\cF_{\Hilb}", dashed, maps to] &  & {(Z_{\Hilb}(X),\sigma^X)} \arrow[dd, "Z_{\Hilb}(f)"] \\
			{} \arrow[rr, "\cF_{\Hilb}", dashed, maps to]                 &  & {}                                   \\
			Y \arrow[rr, "\cF_{\Hilb}", dashed, maps to]                  &  & {(Z_{\Hilb}(Y),\sigma^Y)}                   
		\end{tikzcd}
	\end{center}
	is called the {\em free half-braiding functor} on $\Hilb(\cC)$
	\label{def:freehalfbraidingfunctor}
\end{definition}
\medskip

Denoting by $\cU_{\Hilb}: \cZ \Hilb(\cC) \to \Hilb(\cC)$ the forgetful functor, we have $Z_{\Hilb} = \cU_{\Hilb} \circ \cF_{\Hilb}$, and one could think of establishing an adjunction $\cF_{\Hilb} \dashv \cU_{\Hilb}$ so that $Z_{\Hilb}$ is the monad of this adjunction. The counit of this adjunction will not be bounded, i.e. not a morphism in $\Hilb(\cC)$. The monadic data will be well defined in the algebraic ind-completion, but we have to modify the free-half braiding construction a little.
\medskip

Let $X \in \Vect(\cC)$. For each $U \in \cC$, let $(\beta^X_U)_{a,b}: \bar{a} \odot X \odot a \odot U \to U \odot \bar{b} \odot X \odot b$ be given by
\[
(\beta^X_U)_{a,b} : =  \sum_{\omega \in O(b,a \otimes U)} (\id_U \odot\omega^\vee \odot \id_X\odot \omega^*) (\bar{R}_U \odot \id_{\bar{a}} \odot \id_X \odot \id_{a}).
\]

A simplified verison of the proof of \cite{MR3509018}, Theorem 3.4, shows the following.

\begin{theorem}
	Given $X \in \Vect(\cC)$, $\beta^X := \{ (\beta^X)_{a,b}\}_{a,b \in \Irr(\cC)}$ is a  half-braiding on $Z_{\Vect}(X)$.
	\label{thm:algfreehalfbraiding}
\end{theorem}

Analogous to $\cF_{\Hilb}$, we have the functor 
\[
\cF_{\Vect}: \Vect(\cC) \to \cZ \Vect(\cC)
\]
that takes $X \in \Vect(\cC)$ to $(Z_{\Vect}(X),\beta^X)$.


\subsection{Monadic reconstruction of the algebraic Drinfeld center}

We  now present an algebraic treatment, showing that the algebraic Drinfeld center $\cZ \Vect(\cC)$ is monadic over $\Vect(\cC)$. The statements and proofs in this section follow \cite{MR3674995} very closely. Though we have to be careful with the presence of ind-objects, which are non-dualizable, the adaptations of the arguments are very straightforward. The reader familiar with this reference could skip the proofs in this section without loosing much content.

\begin{definition}
	The unit of $Z_{\Vect}$ is the natural transformation $\eta: \id_{\Vect(\cC)} \to Z_{\Vect}$ defined by 
	\[
	\eta_X: X = \bar{\un} \odot X \odot \un \hookrightarrow \bigoplus_{a \in \Irr(\cC)} \bar{a} \odot X \odot a= Z_{\Vect}(X) .
	\]
\end{definition}

\begin{definition}
	Given $X,Y \in \Vect(\cC)$, define $\partial_{X,Y}: X \odot Y \to Y \odot Z_{\Vect}(X)$ by
	\[
	\partial_{X,Y} = \beta^X_Y \circ (\eta_X \odot \id_Y) .
	\]
\end{definition}

\begin{lemma}
	For $U \in \cC$, 
	$$\partial_{X,U} = (\id_U \odot \rho_{X,U})(\bar{R}_U \odot \id_X \odot \id_U) \ . $$
	\label{lemma:recoveringVirelizierTuraevformula}
\end{lemma}
\begin{proof}
	Fixing $U \in \cC$ and $b \in \Irr(\cC)$, the $b$-th component $X \odot U \to U \odot \bar{b} \odot X \odot b$ of $\beta^X_U(\eta_X \odot \id_U)$ is
	\begin{align*}
		\sum_{a \in \Irr(\cC)} \sum_{ \omega \in O(b,a\odot U)}  (\id_U \odot \omega^\vee \odot \id_X \odot \omega^*)(\bar{R}_U \odot \id_{\bar{a}} \odot \id_X \odot \id_{a})(\eta_X \odot \id_U) \ ,
	\end{align*}
	which equals to 
	\begin{align*}
		\sum_{\omega \in O(b,U)}  (\id_U \odot \omega^\vee \odot \id_X \odot \omega^*)(\bar{R}_U \odot \id_X) \ . 
	\end{align*}
\end{proof}

\begin{lemma}
	Let $X,Y \in \Vect(\cC)$ and $U \in \cC$. If $\xi = \{\xi_U: X \odot U \to U \odot Y \}_{U \in \cC}$ is a natural transformation, then there is a unique $r: Z_{\Vect}(X) \to Y$ such that
	$$\xi_U = (\id_U \odot r) \partial_{X,U} \ \forall \ U \in \cC \ . $$
	\label{lemma:factorizationpropofZpart1}
\end{lemma}
\begin{proof}
	Define 
	\[
	\tilde{\xi}_U = (R_U^* \odot \id_Y)(\id_{\bar{U}} \odot \xi_U): \bar{U} \odot X \odot U \to Y.
	\]
	Then $\tilde{\xi} := \{ \tilde{\xi}_U: C_X(U) \to Y \}_{U \in \cC}$ is a dinatural transformation. By Theorem \ref{thm:definitionmonadZ}, there is a unique $r: Z_{\Vect}(X) \to Y$ such that $\tilde{\xi}_U = r \circ \rho_{X,U}$ for all $U \in \cC$. From the conjugate equation for $U$, it holds that $(\id_U \odot \tilde{\xi}_U)(\bar{R}_U \odot \id_X \odot \id_U) = \xi_U$. But
	\[
	(\id_U \odot \tilde{\xi}_U)(\bar{R}_U \odot \id_X \odot \id_U) = (\id_U \odot r)(\id_U \odot \rho_{X,U})(\bar{R}_U \odot \id_X \odot \id_U) = (\id_U \odot r) \partial_{X,U} ,
	\]
	where the last equality follows from Lemma~\ref{lemma:recoveringVirelizierTuraevformula}.
\end{proof}

\begin{remark}
	Observe that if $\xi = \{\xi_U: X \odot U \to U \odot Y \}_{U \in \cC}$ is a natural transformation between the functors $X \odot (-)$ and $(-) \odot Y: \cC \to \Vect(\cC)$, it has a unique extension to a natural transformation $\xi = \{\xi_K: X \odot K \to K \odot Y \}_{K \in \Vect(\cC)}$. Indeed, taking $r:Z_{\Vect}(X) \to Y$ to be the unique morphism for which $\xi_U = (\id_U \odot r) \partial_{X,U}$ for each $U \in \cC$, the unique extension is given by
	\[
	\xi_K = (\id_K \odot r) \partial_{X,K} ,
	\]
	$K \in \Vect(\cC)$. It is clearly well defined, natural, and since every object in $\Vect(\cC)$ is an inductive limit of objects of $\cC$, it is the unique extension. Then Lemma \ref{lemma:factorizationpropofZpart1} implies that any natural transformation of the form $\xi: X \odot (-) \to (-) \odot Y$ on $\Vect(\cC)$ is determined by a morphism $r: Z_{\Vect}(X) \to Y$, which is determined in turn by the values of $\xi$ on objects of $\cC$.
\end{remark}

\begin{lemma}
	If $X, Y \in \Vect(\cC)$, $V \in \cC$ and $\xi = \{\xi_U: X \odot U \odot V \to U \odot Y \}_{U \in \cC}$ is natural, then there is a unique $r: Z_{\Vect}(X) \odot V \to Y$ such that 
	$$\xi_U = (\id_U \odot r)(\partial_{X,U} \odot \id_V) \ \forall \ U \in \cC \ . $$
	\label{lemma:factorizationpropofZpart2}
\end{lemma}

\begin{proof}
	Define $\tilde{\xi}_U := (\xi_U \odot \id_{\bar{V}})(\id_X \odot \id_U \odot \bar{R}_V): X \odot U \to U \odot Y \odot \bar{V}$. Then $\tilde{\xi} = \{ \tilde{\xi}_U\}_{U \in \cC}$ is natural. By Lemma \ref{lemma:factorizationpropofZpart1}, there is a unique $\tilde{r}:Z_{\Vect}(X) \to Y \odot \bar{V}$ such that $\tilde{\xi}_U = (\id_U \odot \tilde{r})\partial_{X,U}$ for each $U \in \cC$. From the conjugate equations for $V$, by defining $r = (\id_Y \odot R_V^*) (\tilde{r} \odot \id_V)$ we see that
	$$(\id_U \odot r) (\partial_{X,U} \odot \id_V) = \xi_U \ . $$
\end{proof}

\begin{lemma}
	Let $X,Y \in \Vect(\cC)$, and let $\xi: \{ \xi_{U,V}: X \odot U \odot V \to U \odot V \odot Y\}_{U,V \in \cC}$ be natural. There exists a unique $r: Z_{\Vect}^2(X) \to Y$ such that
	$$\xi_{U,V} = (\id_U \odot \id_V \odot r)(\id_U \odot \partial_{Z_{\Vect}(X),V})(\partial_{X,U} \odot \id_V) \ . $$
	Any such natural transformation has a unique extension to a natural transformation $\xi = \{\xi_{K,L}: X \odot K \odot L \to K \odot L \odot Y\}_{K,L \in \Vect(\cC)}$.
	\label{lemma:factorizationpropofZpart3}
\end{lemma}

\begin{proof}
	Fixing $V \in \cC$, $\xi_V:= \{\xi_{U,V}: X \odot U \odot V \to U \odot V \odot Y\}_{U \in \cC}$ is natural in $U \in \cC$. By Lemma \ref{lemma:factorizationpropofZpart2}, there is a unique $\tilde{r}_V: Z_{\Vect}(X) \odot V \to V \odot Y$ such that $\xi_{U,V} = (\id_U \odot \tilde{r}_V)(\partial_{X,U} \odot \id_V)$ for each $U \in \cC$. But since $\xi$ is also natural in the second variable, the collection $\tilde{r} = \{ \tilde{r}_V: Z_{\Vect}(X) \odot V \to V \odot Y\}_{V \in \cC}$ must be natural in $V \in \cC$, by unicity of the $\tilde{r}_V$'s. Applying Lemma \ref{lemma:factorizationpropofZpart1} to $\tilde{r}$, we conclude that there is a unique $r: Z_{\Vect}(Z_{\Vect}(X)) = Z_{\Vect}^2(X) \to Y$ such that $\tilde{r}_V = (\id_V \odot r) \partial_{Z_{\Vect}(X),V}$ for each $V \in \cC$. We have then
	\[
	\xi_{U,V} = (\id_U \odot \tilde{r}_V)(\partial_{X,U} \odot \id_V) = (\id_U \odot \id_V \odot r)(\id_U \odot \partial_{Z_{\Vect}(X),V})(\partial_{X,U} \odot \id_V).
	\]
	By the same argument in the discussion following the proof of Lemma \ref{lemma:factorizationpropofZpart1}, we conclude that $\xi$ extends uniquely to a natural transformations between the endofunctors $X \odot (-) \odot(\bullet) \to (-) \odot (\bullet) \odot Y$ of $\Vect(\cC) \times \Vect(\cC)$.
\end{proof}

\subsection{The adjunction}
Recall the functor $\cF_{\Vect}: \Vect(\cC) \to \cZ \Vect(\cC);\ X \mapsto (Z_{\Vect}(X),\beta^X)$. Let $\cU_{\Vect}: \cZ \Vect(\cC) \to \Vect(\cC)$ be the forgetful functor. Then $\cU_{\Vect} \circ \cF_{\Vect} = Z_{\Vect}$. For simplicity, we shall write $\cF = \cF_{\Vect}$, $Z = Z_{\Vect}$ and $\cU = \cU_{\Vect}$.

\begin{lemma}
	The collection $\{\beta^X_Y: Z(X) \odot Y \to Y \odot Z(X)\}_{X,Y}$ is a natural transformation $Z(-) \odot (\bullet) \to (\bullet) \odot Z(-)$.
	\label{lemma:naturarilityfreehalfbraiding}
\end{lemma}
\begin{proof}
	Given $f \in \Vect(\cC)(X,X')$ and $g \in \Vect(\cC)(Y,Y')$, $Z(f) \in \cZ \Vect(\cC)(Z(X),Z(X'))$, and since $\beta^X$ is a half-braiding, $\beta^X_{Y'}(\id_{Z(X)} \odot g) = (g \odot \id_{Z(X)}) \beta^X_Y$. Therefore
	\begin{align*}
		\beta^{X'}_{Y'}(Z(f) \odot g) & = (\id_{Y'} \odot Z(f))\beta^X_{Y'}(\id_{Z(X)} \odot g) \\
		& = (\id_{Y'} \odot Z(f))(g \odot \id_{Z(X)})\beta^X_Y \ .
	\end{align*}
\end{proof}

Now we introduce a natural transformation that will latter be shown to be the counit of the adjunction $\cF \dashv \cU$. Given $(M,\gamma) \in \cZ \Vect(\cC)$, the naturality of $\{\gamma_Y: M \odot Y \to Y \odot M \}_{Y \in \Vect(\cC)}$ implies, by Lemma \ref{lemma:factorizationpropofZpart1}, the existence of a unique $\varepsilon_{(M,\gamma)}: Z(M) \to M$ such that
$$ \gamma_Y = (\id_Y \odot \varepsilon_{(M,\gamma)})\partial_{M,Y} = (\id_Y \odot \varepsilon_{(M,\gamma)}) \beta^M_Y (\eta_M \odot \id_Y) $$
for each $Y \in \Vect(\cC)$.

\begin{proposition}
	For each $(M,\gamma) \in \cZ \Vect(\cC)$, $$\varepsilon_{(M,\gamma)} \in \cZ \Vect(\cC)((Z(M),\beta^M),(M,\gamma)) \ . $$
\end{proposition}
\begin{proof}
	Since a half-braiding on an ind-object is determined by its values at objects of $\cC$, it suffices to show
	$$(\id_U \odot \varepsilon_{(M,\gamma)})\beta^M_U = \gamma_U (\varepsilon_{(M,\gamma)} \odot \id_U)$$
	for each $U \in \cC$. Due to lemma \ref{lemma:factorizationpropofZpart2} it suffices in turn to prove 
	$$(\id_V \odot (\id_U \odot \varepsilon_{(M,\gamma)})\beta^M_U)(\partial_{M,V} \odot \id_U) = (\id_V \odot \gamma_U(\varepsilon_{(M,\gamma)} \odot \id_U))(\partial_{M,V} \odot \id_U)$$
	for each $V \in \cC$. Now
	\begin{align*}
		(\id_V \odot (\id_U \odot \varepsilon_{(M,\gamma)})\beta^M_U)(\partial_{M,V} \odot \id_U) & = (\id_V \odot \id_U \odot \varepsilon_{(M,\gamma)}) \beta^M_{V \odot U} (\eta_M \odot \id_V \odot \id_U) \\
		& = (\id_{V \odot U} \odot \varepsilon_{(M,\gamma)}) \partial_{M, V \odot U} \\
		& = \gamma_{V \odot U} \\
		& = (\id_V \odot \gamma_U)(\gamma_V \odot \id_U) \\
		& = (\id_V \odot \gamma_U)(\id_V \odot \varepsilon_{(M,\gamma)} \odot \id_U) (\partial_{M,V}\odot \id_U) \ .
	\end{align*}
\end{proof}

\begin{proposition}
	The collection $\varepsilon := \{ \varepsilon_{(M,\gamma)}\}_{(M,\gamma) \in \cZ \Vect(\cC)}$ is a natural transformation $\cF \cU \to \id_{\cZ \Vect(\cC)}$.
	\label{prop:naturalityofvarepsilon}
\end{proposition}

\begin{proof}
	We shall show that, for any $X \in \Vect(\cC)$ and any morphism $f \in \cZ \Vect(\cC)((M,\gamma),(N,\gamma'))$,
	\[
	(\id_X \odot \varepsilon_{(N,\gamma')}Z(f)) \partial_{M,X} = (\id_X \odot f \varepsilon_{(M,\gamma)}) \partial_{M,X}.
	\]
	By Lemma \ref{lemma:factorizationpropofZpart1}, this implies $\varepsilon_{(N,\gamma')}Z(f) = f \varepsilon_{(M,\gamma)}$ and therefore the naturality of $\varepsilon$.
	\begin{align*}
		(\id_X \odot \varepsilon_{(N,\gamma')}Z(f)) \partial_{M,X} & = (\id_X \odot \varepsilon_{(N,\gamma')}Z(f)) \beta^M_X (\eta_M \odot \id_X) \\
		& = (\id_X \odot \varepsilon_{(N,\gamma')})\beta^N_X (Z(f) \eta_M \odot \id_X) \\
		& = (\id_X \odot \varepsilon_{(N,\gamma')})\beta^N_X (\eta_N f \odot \id_X) \\
		& = \gamma'_X (f \odot \id_X) \\
		& = (\id_X \odot f)\gamma_X
		\\
		& = (\id_X \odot f)(\id_X \odot \varepsilon_{(M,\gamma)})\partial_{M,X} \\
		& = (\id_X \odot f \varepsilon_{(M,\gamma)}) \partial_{M,X} \ . 
	\end{align*}
\end{proof}

\begin{lemma}
	For all $X \in \Vect(\cC)$ we have $\varepsilon_{\cF(X)} Z(\eta_X) = \id_{\cF(X)}$, where $\varepsilon_{\cF(X)} = \varepsilon_{(Z(X),\beta^X)}$.
\end{lemma}

\begin{proof}
	We shall show that $(\id_U \odot \varepsilon_{\cF(X)}Z(\eta_X)) \partial_{X,U} = \partial_{X,U}$ for each $U \in \cC$ and invoke Lemma \ref{lemma:factorizationpropofZpart1} to conclude the proof. We have
	\begin{align*}
		(\id_U \odot \varepsilon_{\cF(X)}Z(\eta_X)) \partial_{X,U} & = (\id_U \odot \varepsilon_{\cF(X)}Z(\eta_X)) \beta^X_U (\eta_X \odot \id_U) \\
		&= (\id_U \odot \varepsilon_{\cF(X)}) \beta^{Z(X)}_U (\eta_{Z(X)} \eta_X \odot \id_U) \\
		& = (\id_U \odot \varepsilon_{\cF(X)}) \partial_{Z(X),U}(\eta_X \odot \id_U) \\
		& = \beta^X_U (\eta_X \odot \id_U) \\
		& = \partial_{X,U} \ .
	\end{align*}
\end{proof}

\begin{lemma}
	For all $(M,\sigma) \in \cZ \Vect(\cC)$, it holds $\cU(\varepsilon_{(M,\sigma)})\eta_{\cU(M,\sigma)} = \id_{\cU(M,\sigma)}$.
\end{lemma}
\begin{proof}
	Recall that, for any half-braiding $(N,\gamma)$, $\gamma_\un = \id_N$. In particular $\beta^M_\un = \id_{Z(M)}$. Thus
	\begin{align*}
		\cU(\varepsilon_{(M,\beta)})\eta_{\cU(M,\beta)} = \varepsilon_{(M,\beta)} \beta^M_\un \eta_M = (\id_\un \otimes \varepsilon_{(M,\beta)}) \partial_{M,\un} = \beta_\un = \id_M \ .
	\end{align*}
\end{proof}

\begin{corollary}
	The pair $(\eta,\varepsilon)$ constitutes an adjunction data for $\cF \dashv \cU$.
	\label{cor:adjunctionbetweenFandU}
\end{corollary}

Corollary \ref{cor:adjunctionbetweenFandU} implies that the endofunctor $Z$ of $\Vect(\cC)$ has a monadic structure coming from the adjunction data $\cF \dashv \cU$. We already have a concrete presentation for the unit $\eta$. From the proof of Theorem \ref{thm:definitionmonadZ} we can also extract a concrete presentation for the counit $\varepsilon$. If $(M,\gamma) \in \cZ \Vect(\cC)$, then
\[
\varepsilon_{(M,\gamma)} = \sum_{a \in \Irr(\cC)} (R_{a}^* \odot \id_M)(\id_{\bar{a}} \odot \gamma_{a}) \ , 
\]
which is well defined as a natural transformation between the functors $Z(M)$ and $M$ from $\cC^{\op} \to \Vect$. Indeed, given $U \in \cC$,
\[Z(M)(U) = \bigoplus_{a\in \Irr(\cC)} (\bar{a} \odot M \odot a) (U) .
\]
If $v \in Z(M)(U)$, there is a finite set $I_v \subset \Irr(\cC)$ such that $v \in \bigoplus_{a \in I_v} (\bar{a} \odot M \odot a) (U)$. Then
\[
\varepsilon_{(M,\gamma)}(v) = \sum_{a \in I_v}(R_{a}^* \odot \id_M)(\id_{\bar{a}} \odot \gamma_{a}) (v).
\]

From the general theory of monads and adjunctions, we conclude that $ \cU \circ \cF = Z$ is a monad in $\Vect(\cC)$: its unit is $\eta$, and its multiplication $\mu: Z^2 \to Z$ is given by
\[
\mu_X = \cU(\varepsilon_{\cF(X)}) = \varepsilon_{(Z(X),\sigma^X)} ,
\]
which can be computed, using conjugate equations, to be
$$\mu_X = \sum_{a,b,c \in \Irr(\cC)} \sum_{\omega \in O(c,a \odot b)} (\omega^\vee \odot \id_X \odot \omega^*) \ ,$$
again understood as a natural transformation $Z^2(X) \to Z(X)$ between functors $\cC^{\op} \to \Vect(\cC)$. 

\subsection{Linear isomorphism}

\begin{proposition}
	There is a canonical functor $\Phi: \cZ \Vect(\cC) \to Z \- \Mod_{\Vect(\cC)}$
	\label{prop:functorcentertomodules}
\end{proposition}
\begin{proof}
	If $(M,\gamma) \in \cZ \Vect(\cC)$, then $(M,\varepsilon_{(M,\gamma)}) \in Z \- \Mod_{\Vect(\cC)}$. The proof of Proposition \ref{prop:naturalityofvarepsilon} shows exactly that any morphism $f: (M,\gamma) \to (N,\theta)$ of half-braidings is also a morphism between the $Z$-modules $(M,\varepsilon_{(M,\gamma)})$ and $(N,\varepsilon_{(N,\theta)})$. Define then $\Phi$ by $\Phi(M,\gamma) = (M,\varepsilon_{(M,\gamma)})$ on objects and $\Phi(f) = f$ on morphisms.
\end{proof}

\begin{proposition}
	There is a canonical functor $\Psi: Z \- \Mod_{\Vect(\cC)} \to \cZ \Vect(\cC)$.
	\label{prop:functormodulestocenter}
\end{proposition}
\begin{proof}
	Let $(M,\tau) \in Z \- \Mod_{\Vect(\cC)}$. Given $X \in \Vect(\cC)$, let
	\[
	\beta^\tau_X := (\id_X \otimes \tau) \partial_{M,X} = (\id_X \otimes \tau) \beta^M_X (\eta_M \otimes \id_X).
	\]
	We claim that $\beta^\tau := \{\beta^\tau_X\}_{X \in \Vect(\cC)}$ is a half-braiding on $M$. First we show $\beta^\tau_\un = \id_M$. Since $(Z,\mu,\eta)$ is a monad, 
	\[
	\beta^\tau_\un = \tau \beta^M_\un \eta_M = \tau \eta_M = \id_M.
	\]
	Now we show the braid relation $\beta^\tau_{X \odot Y} = (\id_X \odot \beta^\tau_Y)(\beta^\tau_X \odot \id_Y)$ for arbitrary $X,Y \in \Vect(\cC)$. By definition,
	\begin{align*}
		\beta^\tau_{X \odot Y} & = (\id_X \odot \id_Y \odot \tau) \partial_{M,X \odot Y} \\
		& = (\id_X \odot \id_Y \odot \tau \mu_M) (\id_X \odot \partial_{Z(M),Y})(\partial_{M,X} \odot \id_Y) \\
		& = (\id_X \odot \id_Y \odot \tau Z(\tau)(\id_X \odot \partial_{Z(M),Y})(\partial_{M,X}\odot \id_Y) \\
		& = (\id_X \odot \id_Y \odot \tau)(\id_X \odot \partial_{Z(M),Y})(\id_X \odot \tau \odot \id_Y)(\partial_{M,X}\odot \id_Y) \\
		& = (\id_X \odot \beta^\tau_Y)(\beta^\tau_X \odot \id_Y) \ . 
	\end{align*}
	We show now that $\beta^\tau$ is invertible. Let $U \in \cC$. By naturality and conjugate equations, the morphism
	\[
	\tilde{\beta}^\tau_U := (\bar{R}_U^* \odot \id_M \odot \id_U) (\beta^\tau_{\bar{U}} \odot \id_U)(\id_U \odot \id_M \odot R_U)
	\]
	is an inverse for $\beta^\tau_U$. By naturality, $\beta^\tau$ is determined by the family $\{\beta^\tau_U\}_{U \in \cC}$, and the inverse of $\beta^\tau$ is also defined by its values at objects $U \in \cC$. Thus, $\tilde{\beta} = \{ \tilde{\beta}^\tau_U\}_{U \in \cC}$ has a unique extension to a natural transformation $(-) \odot M \to M \odot (-)$ between endofunctors of $\Vect(\cC)$, and this extension is $(\beta^\tau)^{-1}$.

	Let $f: (M,\tau) \to (N,\upsilon)$ be a morphism of $Z$-modules. We shall show that $f$ is automatically a morphism $(M,\beta^\tau) \to (N,\beta^\upsilon)$ of half-braidings, i.e., that $(\id_X \odot f) \beta^\tau_X = \beta^\upsilon_X (f \odot \id_X)$ for all $X \in \Vect(\cC)$. From the definition of $\beta^\tau$ and the hypothesis that $f$ is a morphism of $Z$-modules, we obtain
	\begin{align*}
		(\id_X \odot f) \beta^\tau_X = (\id_X \odot f) (\id_X \odot \tau) \partial_{M,X} 
		= (\id_X \odot \upsilon) (\id_X \odot Z(f)) \partial_{M,X} \ . 
	\end{align*}
	From the naturality of $\partial_{(-),(\bullet)}$, we have
	\begin{align*}
		(\id_X \odot \upsilon) (\id_X \odot Z(f)) \partial_{M,X}  = (\id_X \odot \upsilon) \partial_{N,X} (f \odot \id_X) = \beta^\upsilon_X (f \odot \id_X) \ . 
	\end{align*}
	
	Hence, defining $\Psi(M,\tau) := (M,\beta^\tau)$ on objects $(M,\tau) \in Z \- \Mod_{\Vect(\cC)}$, and $\Psi(f) = f$ on morphisms, we get a functor $\Psi: Z \- \Mod_{\Vect(\cC)} \to \cZ \Vect(\cC)$. 
\end{proof}

\begin{theorem}
	The functors $\Phi$ and $\Psi$ from propositions \ref{prop:functorcentertomodules} and \ref{prop:functormodulestocenter} are mutually inverses. Hence, there is an isomorphism
	$$\cZ \Vect(\cC) \overset{\text{isom.}}{\simeq} Z \- \Mod_{\Vect(\cC)}  $$
	of linear categories.
	\label{thm:linearisom}
\end{theorem}
\begin{proof}
	Since both functors act identically on morphisms, we just have to show the identities $\Psi \circ \Phi = \id_{\cZ \Vect(\cC)}$ and $\Phi \circ \Psi = \id_{Z \- \Mod_{\Vect(\cC)}}$ on objects. We have
	\begin{align*}
		\Psi \Phi (M,\gamma) & = (M,\sigma^{\varepsilon_{(M,\gamma)}}) \\
		\Phi \Psi (M,\tau) &= (M,\varepsilon_{(M,\beta^\tau)}) \ .
	\end{align*}
	But the identities $\varepsilon_{(M,\beta^\tau)} = \tau$ and $\beta^{\varepsilon_{(M,\gamma)}} = \gamma$ follow from the definition of the counit $\varepsilon$.
\end{proof}

\subsection{Bimonadic structure of $Z$}

Now we shall use lemmas \ref{lemma:factorizationpropofZpart1}, \ref{lemma:factorizationpropofZpart2} and \ref{lemma:factorizationpropofZpart3} to give to $Z$ the structure of a bimonad. Let $X_1,X_2 \in \Vect(\cC)$. The collection 
\[
\{(\partial_{X_1,K} \odot \id_{Z(X_2)})(\id_{X_1} \odot \partial_{X_2,K}): X_1 \odot X_2 \odot K \to K \odot Z(X_1) \odot Z(X_2)\}_{K \in \Vect(\cC)}
\]
is natural in $K \in \Vect(\cC)$. By Lemma \ref{lemma:factorizationpropofZpart1}, there is a unique $Z_2(X_1,X_2):Z(X_1 \odot X_2) \to Z(X_1) \odot Z(X_2)$ such that
\[
(\partial_{X_1,U} \odot \id_{Z(X_2)})(\id_{X_1} \odot \partial_{X_2,U}) = (\id_U \odot Z_2(X_1,X_2)) \partial_{X_1 \odot X_2,U} 
\]

\begin{lemma}
	For every $X_1, X_2 \in \Vect(\cC)$, 
	\[
	Z_2(X_1,X_2) = \sum_{a \in \Irr(\cC)} \id_{\bar{a}} \odot \id_{X_1} \odot \bar{R}_a \odot \id_{X_2} \odot \id_{a} .
	\]
\end{lemma}
\begin{proof}
	Observe first that the right-hand-side is well defined as a natural transformation. Indeed, for each $c \in \Irr(\cC)$, 
	\[ Z(X_1 \odot X_2)(c) = \bigoplus_{a \in \Irr(\cC)} \Vect(\cC)(c, \bar{a} \odot X_1 \odot X_2 \odot a),
	\]
	and therefore for each $v \in Z(X_1 \odot X_2)(c)$ there are only finitely many $a \in \Irr(\cC)$ for which the image of $v$ under
	\[
	\id_{\bar{a}} \odot \id_{X_1} \odot \bar{R}_a \odot \id_{X_2} \odot \id_{a} : \Vect(\cC)(c, \bar{a} \odot X_1 \odot X_2 \odot a) \to \Vect(\cC)(c, \bar{a} \odot X_1 \odot a \odot \bar{a} \odot X_2 \odot a)
	\]
	is non-zero. By naturarility, it suffices to prove 
	\[
	(\partial_{X_1,U} \odot \id_{Z(X_2)})(\id_{X_1} \odot \partial_{X_2,U}) = (\id_U \odot Z_2(X_1,X_2)) \partial_{X_1 \odot X_2,U} 
	\]
	when $K = c \in \Irr(\cC)$. This is a consequence of the following observation: given $Y \in \Vect(\cC)$, 
	\[
	\partial_{Y,c} = \bar{R}_c \odot \id_Y \odot \id_c,
	\]
	upon identification of $\bar{c} \odot Y \odot c$ as a subobject of $\cF(Y) = \bigoplus_a \bar{a} \odot Y \odot a$. Hence,
	\begin{align*}
	(\partial_{X_1,U} \odot \id_{Z(X_2)})(\id_{X_1} \odot \partial_{X_2,U}) &= \bar{R}_c \odot \id_{X_1} \odot \bar{R}_c \odot \id_{X_2} \odot c \\
	& = (\id_c \odot (\id_{\bar{c}} \odot \id_{X_1} \odot \bar{R}_c \odot \id_{X_2} \odot \id_c))(\bar{R}_c \odot \id_{X_1} \odot \id_{X_2} \odot \id_c)\\
	& = (\id_c \odot Z_2(X_1,X_2)) \partial_{X_1 \odot X_2,c} .
	\end{align*}
\end{proof}	

From the above Lemma we deduce that $Z_2$ is natural in both variables.
\medskip

Looking at $\id_{\Vect(\cC)}$ as a natural transformation $\id_K: K \simeq \un \odot K \to K \odot \un \simeq K$, Lemma \ref{lemma:factorizationpropofZpart1} implies the existence of a unique $Z_0: Z_{\Vect}(\un) \to \un$ such that 
$$\id_K = (\id_K \odot Z_0) \partial_{\un,K}$$
for each $K \in \Vect(\cC)$.

\begin{lemma}
	The unique morphism $Z_0: Z(\un) \to \un$ saitsfying
	\[
	\id_K = (\id_K \odot Z_0) \partial_{\un,K}
	\]
	for every $K \in \Vect(\cC)$ is given by
	\[
	Z_0 = \sum_{a \in \Irr(\cC)} R_a^* : \bigoplus_{a \in \Irr(\cC)} \bar{a} \odot a \to \un .
	\]
\end{lemma}
\begin{proof}
	Given $c \in \Irr(\cC)$ and $v \in Z(\un)(c) \simeq \bigoplus_a \cC(c, \bar{a} \odot a)$, there are only finitely many $a \in \Irr(\cC)$ for which the image of $v$ under $R_a^*: \cC(c,\bar{a} \odot a) \to \cC(c,\un)$ is non-zero. Hence $Z_0$ is well-defined. By naturality, it suffices to show that
	\[
	\id_c = (\id_c\odot Z_0) \partial_{\un,c}
	\]
	for every $c \in \Irr(\cC)$. This follows from the argument used in the computation of $Z_2$, as can be easily checked.
\end{proof}

If $(T,T_2,T_0)$ is an op-lax endofunctor of a  tensor category $\cD$, then $T^2 = T \circ T$ can be equipped with a canonical op-lax structure, the op-unitor given by $T_0 \circ T(T_0)$ and the op-tensorator given by
\[
T^2(X \otimes Y) \overset{T(T_2)}{\longrightarrow} T(T(X) \otimes T(Y)) \overset{T_2}{\longrightarrow} T^2(X) \otimes T^2(Y) .
\]

\begin{definition}\label{def:bimonad}
	Suppose $(T,\mu, \eta)$ is a monad over a tensor category $\cD$. If $(T_2,T_0)$ is an op-lax structure on $T$, the tuple $(T,\mu,\eta,T_2,T_0)$ is called a {\em bimonad} if $\mu$ and $\eta$ are op-lax natural transformations.
\end{definition}

The proof of the following Proposition follows easily from diagrammatical calculus.

\begin{proposition}
	The pair $(Z_2,Z_0)$ is an op-lax  structure on $Z$, turning $(Z,\mu,\eta,Z_2,Z_0)$ into a bimonad in $\Vect(\cC)$.
	\label{prop:Zisbimonad}
\end{proposition}

By making use of the bimonadic structure on $Z$ given by $(Z_2,Z_0)$, the category $Z \- \Mod_{\Vect(\cC)}$ can be given a tensor structure such that the forgetful functor $Z \- \Mod_{\Vect(\cC)} \to \Vect(\cC)$ is strictly monoidal. Indeed, if $(M,\tau)$ and $(N,\upsilon)$ are $Z$-modules,
\[
Z(M \odot N) \overset{Z_2}{\longrightarrow} Z(M) \odot Z(N) \overset{\tau \odot \upsilon}{\longrightarrow} M \odot N
\]
defines a $Z$-module structure on $M \odot N$.

The Lemma below is also easily proven via diagrammatical calculus.

\begin{lemma}
	Let $(M,\tau)$ and $(N,\upsilon)$ be $Z$-modules. Then
	\[
	(M \odot N, \beta^{(\tau \odot \upsilon) \circ Z_2}) = (M,\beta^\tau) \odot (N,\beta^\upsilon) \in \cZ \Vect(\cC) .
	\]
\end{lemma}

\begin{corollary}
	The linear isomorphism $Z \- \Mod_{\Vect(\cC)} \simeq \cZ \Vect(\cC)$ is an isomorphism of tensor categories.
\end{corollary}

\section{Unitary modules for C$^*$-algebra objects}

The previous Section shows that the algebraic Drinfeld center $\cZ \Vect(\cC)$ of a unitary tensor category is monoidally equivalent to the category of modules for a monad in $\Vect(\cC)$. In the next Section, we will observe that the same category can be described as bimodules for a concrete $*$-algebra object in $\Vect(\cC^{mp} \boxtimes \cC)$. As with $*$-algebras in $\Vect$, one can ask what is a {\em unitary bimodule} for a $*$-algebra object in $\Vect(\cC^{mp} \boxtimes \cC)$. The present Section is devoted to answering this question, and later it will be shown that it allows for a characterization of unitary half-braidings which is internal to $\Hilb(\cC^{mp} \boxtimes \cC)$.

\subsection{Bounded operators on Hilbert space objects}

\begin{definition}
    For $H \in \Hilb(\cC)$, let $\bB^l_H,\bB^r_H \in \WAlg(\cC)$ be given by
    \[  \bB^l_H(a) := \Hilb(\cC)(a \otimes H,H)\ \ \text{and} \ \ \bB^r_H(a) := \Hilb(\cC)(H \otimes a, H)  , \]
    with multiplications $\mu^l$ and $\mu^r$ given by
    \[
    \mu^l (f \odot g) := f \circ (\id_a  \otimes g)  \ \ , \ \  \mu^r(f' \odot g') := g \circ (f \otimes \id_b) \ ,
    \]
    for $(f,g) \in \bB^l_H(a) \times \bB^l_H(b)$ and $(f',g') \in \bB^r_H(a) \times \bB^r_H(b)$, and $*$-structures given by
    \[
    j^l(f) := \Phi^l_a(f)^* \ \ , \ \ j^r(f') := \Phi^r_a(f')^* \ ,
    \]
    using the Frobenius reciprocity maps. That $\bB^l_H$ and $\bB^r_H$ are W$^*$-algebra objects is seen by realizing that they correspond to the cyclic right $\cC$-module W$^*$-category and the cyclic left $\cC$-module W$^*$-category generated by $H$, respectively.
    
    Similarly for $V \in \Vect(\cC)$, we write $\cL^ l_V$ for the the algebra object in $\Vect(\cC)$ given by
    \[
    \cL^ l_V(-) = \Vect(\cC) ((-) \odot V, V) \ , 
    \]
    \label{def:right-leftboundedoperators}
    with algebra structure analogous to that of $\bB^l_H$. The algebra object $\cL^r_V$ is defined similarly.
\end{definition}

\begin{lemma}
Let $H$ be a Hilbert space object, and let $a \in \Irr{\cC}$. Then
    \[ \bB^l_H(a) \simeq \prod_{a_1,a_2 \in \Irr(\cC)}^{\ell^\infty} \cC(a_1,a \otimes a_2) \otimes \bB(H(a_2),H(a_1)) \ , \] \label{lemma:presentationrightboundedoperators}
    where the norm on the tensor factor $\cC(a_1,a \otimes a_2) $ is
    \[
    \| u \| = \| (\tr_a \otimes \id_{a_2})(uu^*)\|^{1/2} \ .
    \]
\end{lemma}

\begin{proof}
The proof is a direct computation:
    \begin{align*}
        \bB_H^l(a) &= \Hilb(\cC) (a \otimes H,H)  = \prod_{a_1 \Irr(\cC)}^{\ell^\infty} \bB((a \otimes H)(a_1),H(a_1)) \simeq \prod_{a_1 \in \Irr(\cC)}^{\ell^\infty} \bB(H(\bar{a} \otimes a_1),H(a_1))\\ & \overset{(!)}{\simeq} \prod_{a_1,a_2 \in \Irr(\cC)}^{\ell^\infty} \bB(\cC(a_2,\bar{a} \otimes a_1) \otimes H(a_2),H(a_1))  \simeq \prod_{a_1,a_2 \in \Irr(\cC)}^{\ell^\infty} \overline{\cC(a_2,\bar{a} \otimes a_1)} \otimes \bB(H(a_2),H(a_1)) \ .
    \end{align*}
In the right hand side of $(!)$, $\cC(a_2,\overline{a} \otimes a_2)$ is the Hilbert space with norm determined by  
\[
\| v\| ^2 \id_{a_2} = v^* \circ v \ .
\]
We have canonical isomorphisms
\[ \overline{\cC(a_2,\bar{a} \otimes a_1)} \simeq \cC(\overline{a} \otimes a_1, a_2) \simeq \cC(a_1, a \otimes a_2) \ , \]
obtained by composing the $*$-operation, i.e. the dagger strcuture, with Frobenius reciprocity map $\Phi^l_{\overline{a}}$, which gives the claimed norm on $\cC(a_1, a \otimes a_2)$. Therefore,
    \[\bB_H^l(a) \simeq \prod_{a_1,a_2 \in \Irr(\cC)}^{\ell^\infty} \cC(a_1,a \otimes a_2) \otimes \bB(H(a_2),H(a_1))\]
\end{proof}

Recall that $\For$ denotes the forgetful functor $\Hilb(\cC) \to \Vect(\cC)$.

\begin{lemma}
    Given $V \in \Vect(\cC)$, there is a canonical left action of $\cL_V^l$ on $V$. If $H \in \Hilb(\cC)$, there is a canonical left action of $\bB^l_H$ on $\For(H)$.
    \label{lemma:canonicalactionrightboundedoperators}
\end{lemma}
\begin{proof}
The proof of the first claim is identitical to the proof of the second. We only explain the latter. For simplicity, we shall ommit the notation for the forgetful functor $\For: \Hilb(\cC) \to \Vect(\cC)$. There is a canonical linear map
\[ \bB^l_H(a_1) \odot H(a_2) \to H(a_1 \otimes a_2) . \]
Indeed, writing $\bB^l_H(a_1) \odot H(a_2) = \Hilb(\cC)(a_1 \otimes H,H) \odot \Hilb(\cC)(a_2,H)$ and $H(a_1 \otimes a_2) = \Hilb(\cC)(a_1 \otimes a_2, H)$, the above map is given by
    \[ f \odot g \mapsto f \circ (\id_{a_1} \otimes g) \ . \]
    Consider the composition
    \begin{align*}
    \cC(a, a_1 \otimes a_2) \odot \bB^l_H(a_1) \odot H(a_2) &\to \cC(a,a_1 \otimes a_2) \odot H(a_1 \otimes a_2)\to H(a) \\
    v \odot f \odot g &\mapsto v \odot f \circ (\id_{a_1} \otimes g)  \mapsto H(v) \left( f \circ (\id_{a_1} \otimes g) \right). 
    \end{align*}
    Now, since
    \[ (\bB^l_H \odot H) (a) \simeq \bigoplus_{a_1,a_2} \cC(a,a_1 \otimes a_2) \odot \bB^l_H(a_1) \odot H(a_2) \ , \]
    the maps defined above assemble into a morphism $\bB^l_H \odot H \to H$ in $\Vect(\cC)$.

    Using the generic notation $\triangleright: \bB^l_H \odot H \to H$ for the morphism defined above, and denoting by $\mu: \bB^l_H \odot \bB^l_H \to \bB^l_H$ the multiplication morphism for simplicity, we want to show that the diagram
    \begin{center}
\begin{tikzcd}
\bB^l_H \odot \bB^l_H \odot H \arrow[rr, "\mu \odot \id_H"] \arrow[dd, "\id_{\bB^l_H} \odot \triangleright"'] &  & \bB^l_H \odot H \arrow[dd, "\triangleright"] \\
                                                                                                      &  &                                              \\
\bB^l_H \odot H \arrow[rr, "\triangleright"']                                                         &  & H                                           
\end{tikzcd}
    \end{center}
    commutes. Analyzing this diagram component-wisely, one concludes that its commmutativity is equivalent to the commutativity of the diagram
    \begin{center}
\begin{tikzcd}
{\Hilb(\cC)(a_1 \otimes H,H) \odot \Hilb(\cC)(a_2 \otimes H, H) \odot \Hilb(\cC)(a_3,H)} \arrow[rr, "\mu \odot \id_Hd"] \arrow[dd, "\id_{\bB^l_H} \odot \triangleright"'] &  & {\Hilb(\cC)(a_1a_2 \otimes H,H) \odot \Hilb(\cC)(a_3,H)} \arrow[dd, "\triangleright"] &  & f \odot g \odot h \arrow[rr, maps to] \arrow[dd, maps to]  &  & f \circ (\id_{a_2} \otimes g) \odot h \arrow[dd, maps to]   \\
                                                                                                                                                                    &  &                                                                                         &  &                                                            &  &                                                             \\
{\Hilb(\cC)(a_1,H) \odot \Hilb(\cC)(a_2a_3,H)} \arrow[rr, "\triangleright"']                                                                                      &  & {\Hilb(\cC)(a_1a_2a_3,H)}                                                              &  & f \odot (g \circ (\id_{a_2} \otimes h) \arrow[rr, maps to] &  & f \circ (\id_{a_2} \otimes g) (\id_{a_1} \otimes \id_{a_2})
\end{tikzcd}  
    \end{center}
which can be checked directly:
    \begin{center}
        \begin{tikzcd}
            f \odot g \odot h \arrow[rr, maps to] \arrow[dd, maps to]  &  & f \circ (\id_{a_1} \otimes g) \odot h \arrow[dd, maps to] \\
            & & \\
            f \odot (g \circ (\id_{a_2} \otimes h) \arrow[rr, maps to] &  & f \circ (\id_{a_2} \otimes g) \circ (\id_{a_1} \otimes \id_{a_2} \otimes h)
        \end{tikzcd}
    \end{center}
\end{proof}

\subsection{Unitary modules}

Let $(\bA, \mu^\bA)$ be an algebra object in $\Vect(\cC)$. Let $V \in \Vect(\cC)$, and let $\alpha: \bA \odot V \to V$ be a left $\bA$-module structure on $V$.  
The maps
\[
\bA(a_1) \ni f \mapsto \alpha \circ (f \odot \id_H) \in \Vect(\cC)(a_1 \odot H, H)
\]
induce a morphism $\tilde{\alpha}: \bA \to \cL_V^l$ in $\Vect(\cC)$. The associativity of $\alpha$ is then easily checked to be equivalent to $\tilde{\alpha}$ being an algebra morphism. Conversely, invoking the canonical action of $\cL_V^l$ on $V$, any algebra morphism $\tilde{\alpha}: \bA \to \cL_V^l$ induces an action $\alpha$ of $\bA$ on $V$. These constructions are mutually inverses, and thus we have the following Lemma.

\begin{lemma}
	There is a bijective correspondence between left $\bA$-module structures $\alpha$ on $H$ and algebra morphisms $\tilde{\alpha}:\bA \to \cL^ l_H$.
	\label{lemma:correspondenceleftmoduleandalgebramorphism}
\end{lemma}

\begin{definition}
	Let $(\bA,\mu^\bA, j^\bA)$ be a $*$-algebra object, and let $H \in \Hilb(\cC)$. A left $\bA$-module structure $\alpha$ on $\For(H)$ is a {\em unitary} left $\bA$-module structure on $H$ if the corresponding algebra morphism $\tilde{\alpha}$ factors through a $*$-algebra morphism $\tilde{\alpha}: \bA \to \bB^l_H$, along the embedding $\bB^l_H \hookrightarrow \cL^ l_H$. 
\end{definition}

\begin{remark}
	This definition of unitary modules was already used in \cite{MR3687214}, Section 5. Here, we make explain how it relates to the standard notion of modules for algebra objects.
\end{remark}

\medskip

The proof of the following Lemma is straighforward.

\begin{lemma}
	Let $(H,\alpha)$ be a unitary left module for the $*$-algebra object $\bA$. Given $K \in \Hilb(\cC)$, the $W^*$-algebra object embedding
	\[
	\Hilb(\cC)((-) \otimes H,H) \to \Hilb(\cC)((-) \otimes H \otimes K, H \otimes K)
	\]
	given by the tensor structure induces a unitary left $\bA$-module structure on $H \otimes K$. 
	\label{lemma:amplificationunitarymodules}
\end{lemma}

\begin{definition}
	Given a $*$-algebra object $\bA$, a morphism $f: (H,\alpha) \to (K,\beta)$ between unitary $\bA$-modules is a morphism $f \in \Hilb(\cC)$ such that $f \circ \alpha = \beta \circ (\id_\bA \odot f)$. There is a category $\LMod_{\Hilb(\cC)}(\bA)$ whoose objects are unitary left $\bA$-modules and morphisms are unitary $\bA$-module morphisms. 
\end{definition} 

The previous Lemma says that $\LMod_{\Hilb(\cC)}(\bA)$ is a right $\Hilb(\cC)$-module category, and it is straightforward to check that it has a canonical structure of a right $\Hilb(\cC)$-module W$^*$-category.

\section{The Canonical $W^*$-Algebra}

Every unitary tensor category $\cC$ is canonically a right $(\mpcC \boxtimes \cC)$-module C$^*$-category, where the action is given by
\[ c \triangleleft (a_1 \boxtimes a_2) := a_1 \otimes c \otimes a_2 \ . \]
The tensor unit $\un$ is moreover a cyclic object for this action.

\begin{definition}
    The Canonical W$^*$-algebra object of a unitary tensor category $\cC$ is the W$^*$-algebra object $\cS$ in $\Vect(\mpcC \boxtimes \cC)$ corresponding to $(\cC, \un_\cC)$ as a cyclic right $(\mpcC \boxtimes \cC)$-module W$^*$-category. In particular, the functor
    \begin{align*}
        \cM_\cS & \to \cC \\
        \cS  \odot (a \boxtimes b) & \mapsto ab
    \end{align*}
    is an equivalence of right $(\mpcC \boxtimes \cC)$-module W$^*$-categories, where $\cM_\cS$ is the Cauchy completion of the category of free $\cS$-left modules of the form $\cS \odot (a \boxtimes b)$, with $a,b \in \cC$.
    \label{def:SEalgebra}
\end{definition}

\begin{lemma}
    In the context of Definition \ref{def:SEalgebra}, 
    \[ \cS \simeq \bigoplus_{c \in \Irr(\cC)} \overline{c} \boxtimes c \ . \]
    \label{lem:explicitpresentationSEalgebra}
\end{lemma}

\begin{lemma}
    There are canonical isomorphisms
    \[ \cS \odot (\un \boxtimes a) \simeq \cS \odot (a \boxtimes \un) \  \]
    in $\cM_\cS$, natural in $a \in \cC$.
    \label{lem:mirrorpropertySE}
\end{lemma}
\begin{proof}
    Under the functor $\cM_\cS \to \cC$ of Definition \ref{def:SEalgebra}, these objects have the same image. 
\end{proof}
\medskip

\begin{corollary}
    The $(\mpcC \boxtimes \cC)$-module W$^*$-category $\cM_\cS$ is equivalent of the Cauchy completion of the full subcategory generated by $\cS$-modules of the form $\cS \odot (\un \boxtimes a)$, with $a \in \cC$. 
\end{corollary}

\begin{proof}
    For all $a,b \in \cC$,
    \[ \cS \odot (a \boxtimes b) \simeq \cS \odot (a \boxtimes \un) \odot (\un \boxtimes b) \simeq \cS \odot (\un \boxtimes ab) \]
    by Lemma \ref{lem:mirrorpropertySE}.
\end{proof}

\begin{lemma}
The canonical equivalence $\cC \simeq \cM_\cS$ of right $(\mpcC \boxtimes \cC)$-module categories extends to an equivalence 
\[ \Vect(\cC) \simeq \LMod_{\Vect(\cC^{mp} \boxtimes \cC)}(\cS) \  \]
of $\Vect(\mpcC \boxtimes \cC)$-module categories.
\end{lemma}

\begin{proof}
An object $a \in \cC$ corresponds to $\cS \odot (\un \boxtimes a)$. If $W \in \Vect(\cC)$, 
\[ W \simeq \bigoplus_{a \in \Irr(\cC)} a \odot W(a) \ , \]
the corresponding left $\cS$-module is
\[\cS_W := \bigoplus_{a \in \Irr(\cC)} \cS \odot (\un \boxtimes a) \odot W(a) \ . \]
It remains to show that every left $\cS$-module in $\Vect(\cC^{mp}\boxtimes\cC)$ decomposes as a direct sum as above. Recall the adjunction between the free $\cS$-module functor and the forgetful functor from $\LMod_{\Vect(\cC^{mp} \boxtimes \cC)}(\cS) \to \Vect(\cC^{mp}\boxtimes \cC)$: given a left $\cS$-module $V$ in $\Vect(\cC^{mp} \boxtimes \cC)$,
\[ \Hom_{\LMod_{\Vect(\cC^{mp} \boxtimes \cC)}(\cS)}(\cS \odot (b \boxtimes c), V) \simeq \Hom_{\Vect(\cC^{mp} \boxtimes \cC)}(b \boxtimes c, V) \ . \]
This shows that the identity morphism of $V$ can be written as a sum of left $\cS$-module morphisms $S \odot (b \boxtimes c) \to V$. i.e.,
\[ V \simeq \bigoplus_{b,c} \left[ \cS \odot (b \boxtimes c) \right] \odot \Hom_{\LMod(\cS)}(\cS \odot (b \boxtimes c), V) \ . \]
Since
\[ \cS \odot (b \boxtimes c) \simeq \cS \odot(b \boxtimes \un) \odot(\un \boxtimes c) \simeq \cS \odot (\un \boxtimes b) \odot (\un \boxtimes c) \simeq \cS \odot (\un \boxtimes bc) \ ,\]
which are $\cS$-module morphisms by Lemma \ref{lem:mirrorpropertySE}, we are done.
\end{proof}

\begin{lemma}
    For each $a \in \Irr(\cC)$, $\cS_a:= \cS \odot ( \un \boxtimes a)$ has a canonical structure of a Hilbert space object in $\Vect(\cC^{mp} \boxtimes \cC)$.
\end{lemma}
 \begin{proof}
 Indeed,
 \begin{align*}
\cS_a (b \boxtimes c) & = \Hom_{\Vect(\cC^{mp} \boxtimes \cC)} (b \boxtimes c, \cS \odot (\un \boxtimes a))  \simeq \bigoplus_{d \in \Irr(\cC)} \Hom_{\cC^{mp} \boxtimes \cC} (b \boxtimes c, \bar{d} \boxtimes da) \\
& \simeq \bigoplus_{d \in \Irr(\cC)} \cC(b,\bar{d}) \otimes \cC(c \bar{a},d) \simeq  \bigoplus_{d \in \Irr(\cC)} \cC(b,\bar{d}) \otimes \cC(\overline{d},a \overline{c}) \simeq \cC(b,a \overline{c}) \simeq \cC(bc,a) \ .
\end{align*}
The Hilbert space structure of $\cS_a (b \boxtimes c)$ is the one corresponding to the Hilbert space structure of $\cC(bc,a)$ through the isomorphisms above.
\end{proof}

\begin{corollary}
    If $H \in \Hilb(\cC)$, $\cS_H$ is given by
    \[
    \cS_H(b \boxtimes c) = H(bc) \simeq \Hilb(\cC)(bc, H) \ .
    \]
\end{corollary}
\begin{proof}
The proof is a straighforward computation:
    \[ \cS_H(b \boxtimes c) \simeq \bigoplus_{a \in \Irr(\cC)} \cS_a(b \boxtimes c) \odot H(a) \simeq \bigoplus_{a \leq bc}\cC(bc,a) \odot H(a) \simeq H(bc) \ ,\]
    where the notation $a \leq bc$ means that $a$ is a subobject of $bc$.
\end{proof}

\begin{corollary}
	For any vector space object $H$, $\cS_H$ is completely determined by its values on objects of the form $\un \boxtimes b$, with $b \in \Irr(\cC)$. Thus, there is a bijective correspondence between isomorphism classes of Hilbert space object structures on $H$ and isomorphism classes of Hilbert space object structures on  $\cS_H$.
	\label{cor:equivalencebetweenHilbertspacestronHandonS_H}
\end{corollary}
\medskip

The goal now is to analyze the free left $\cS$-module structure on $\cS_a$ and prove that it is unitary. The left $\cS$-action on $\cS_\un$ corresponds to an algebra object morphism $\Gamma: \cS \to \cL^l_{\cS_\un}$. The only non-zero fibers of $\cS$ are $\cS(\bar{a} \boxtimes a)$, with $a \in \Irr(\cC)$. Thus, $\Gamma$ is determined by linear maps $\cS(\bar{a} \boxtimes a) \to \cL^l_{\cS_\un} (\bar{a} \boxtimes a)$. Direct computation shows that 
\[
\cL^l_{\cS_\un} (\bar{a} \boxtimes a) = \Vect(\cC^{mp} \boxtimes \cC)((\overline{a} \boxtimes a) \odot \cS_\un, \cS_\un) \simeq \prod_{a_1,a_2}  \cL \left( \cC(\overline{a}_1, \overline{a}_2 \overline{a}) \odot \cC(a_1, a a_2) \odot \cS(\bar{a}_2 \boxtimes a_2), \cS(\overline{a_1} \boxtimes a_1) \right)  \ .
\]
Let $\iota_a: \overline{a} \boxtimes a \to \cS$ be the embedding corresponding to the localization $\cS \to \overline{a} \boxtimes a$. Then there is an identification
\[
\cS(\overline{a} \boxtimes a) \simeq \C \ \iota_a \ .
\]
Thus, the algebra morphism $\Gamma: \cS \to \cL^l_{\cS_\un}$ corresponding to the free left $\cS$-action on $\cS_\un$, when written in components
\[
\Gamma_{a_1,a_2}: \cS(\overline{a} \boxtimes a) \to  \cL \left( \cC(\overline{a}_1, \overline{a}_2 \overline{a}) \odot \cC(a_1, a a_2) \odot \cS(\bar{a}_2 \boxtimes a_2), \cS(\overline{a_1} \boxtimes a_1) \right) \ , 
\]
is completely determnied by $\Gamma_{a_1,a_2}(\iota_a)$.

\begin{lemma}
	The algebra morphism $\Gamma: \cS \to \cL^l_{\cS_\un}$ corresponding to the free left $\cS$-action on $\cS_\un = \cS$ is given by
	\[
	\Gamma_{a_1,a_2}(\iota_a) (u \odot v \odot \iota_{a_2}) = \langle u^{\vee * }, v \rangle \iota_{{a_1}} \ .
	\]
	In particular,
	\[
	\| \Gamma_{a_1,a_2}(\iota_a)  \| = 1 \ ,
	\]
	for all $a_1,a_2 \in \Irr(\cC)$.
	\label{lemma:algebrahomassociatedtofreeSaction}
\end{lemma}

\begin{proof}
	Let $u \in \cC(\overline{a}_1, \overline{a}_2 \overline{a})$ and $v \in \cC(a_1, a a_2)$. The free left action of $\cS$ on $\cS_\un$ is the multiplication of $\cS$. Thus, 
	\begin{align*}
		\Gamma_{a_1,a_2}(\iota_a) (u \odot v \odot \iota_{a_2}) & = \mu^\cS \circ (\iota_a \odot \iota_{a_2}) \circ (u \boxtimes v) = \sum_{a_3} \sum_{w \in O(a_3,a a_2)} \iota_{a_3} \circ ( (w^\vee)^* \boxtimes w^*)  \circ (u \boxtimes v) \\
		& = \sum_{w \in O(a_1,a a_2)} \iota_{a_1} \circ ( (w^\vee)^* \boxtimes w^*)  \circ (u \boxtimes v) \\
		& = \iota_{a_1}  \circ \sum_{w \in O(a_1,a a_2)} ( (w^\vee)^* \boxtimes w^*)  \circ (u \boxtimes v)  =  \langle u^{\vee * }, v \rangle \iota_{{a_1}} \ .
	\end{align*}
	
	Since 
	\[
	\| \Gamma_{a_1,a_2}(\iota_a) (u \odot v \odot \iota_{a_2})  \| = | \langle u^{\vee * }, v \rangle| \| \iota_{{a_1}} \| = \| u \otimes v \otimes \iota_{a_2} \| \ ,
	\]
	it follows that $\Gamma_{a_1,a_2} (\iota_a)\| \geq 1$ but an application of the Cauchy-Schwarz inequality shows that $\| \Gamma_{a_1,a_2} (\iota_a)\| \leq 1$.
\end{proof}

\begin{proposition}
	Let free left $\cS$-action on $\cS_\un$ is unitary.
\end{proposition}
\begin{proof}
	Lemma \ref{lemma:algebrahomassociatedtofreeSaction} shows that $\Gamma(\iota_a)$ lies in $\bB_{\cS_\un}^l(\bar{a} \boxtimes a)$. It remains to show that $\Gamma$ intertwines the $*$-structures of $\cS$ and $\bB_{\cS_\un}^l$. The $*$-structure of $\bB^l_{\cS_\un}$ sends $ \langle u^{\vee * }, v \rangle \iota_{{a_1}}$ to $\overline{\langle u^{\vee * }, v \rangle } \iota_{a_2}$, which coincides with $\Gamma_{a_2,a_1}(\iota_{\bar{a}})((v^\vee)^* \odot (u^\vee)^* \odot \iota_{a_1})$.
\end{proof}

\begin{corollary}
	For every $H \in \Hilb(\cC)$, $\cS_H$ has a unitary left $\cS$-module structure induced from $\cS_\un$.
\end{corollary}
\begin{proof}
	As a Hilbert space object, $\cS_H \simeq \cS_\un \otimes (\un \boxtimes H)$. By Lemma \ref{lemma:amplificationunitarymodules}, there is an induced unitary left $\cS$-module structure on $\cS_H$.
\end{proof}

\begin{proposition}
	There is an equivalence
	\[ 
	\Hilb(\cC) \simeq \LMod_{\Hilb(\cC^{mp} \boxtimes \cC)} (\cS) 
	 \]
	 of W$^*$-categories.
	 \label{prop:Hilb(C)simequnitaryleftSmodules}
\end{proposition}
\begin{proof}
	The previous proposition shows that $H \mapsto \cS_H$ can be understood as a functor $\Hilb(\cC) \to \LMod_{\Hilb(\cC^{mp} \boxtimes \cC)}(\cS)$. Now, for a Hilbert space object $\cK \in \Hilb(\cC^{mp} \boxtimes \cC)$, the structure of a unitary left $\cS$-module on $\cK$ is a $*$-algebra object morphism $\cS \to \bB^l_{\cK}$. Writing $\cK \simeq \cS_H$ for $H \in \Hilb(\cC)$ given by
	\[ H(a)  = \bigoplus_{b,c \in \Irr(\cC)}^{\ell^2} \cC(a,bc) \otimes \cK(b \boxtimes c) \ , \]
	we conclude that the equivalence $\Vect(\cC) \simeq \LMod_{\Vect(\cC^{mp} \boxtimes \cC)}(\cS)$ induces an equivalence $\Hilb(\cC) \simeq \LMod_{\Hilb(\cC^{mp} \boxtimes \cC)} (\cS)$.
\end{proof}

\begin{remark}
	By Proposition \ref{prop:Hilb(C)simequnitaryleftSmodules}, the category of unitary left $\cS$-modules can be equipped with a unitary tensor structure. Up to unitary monoidal isomorphism, it is determined by
	\[\cS_{H_1} \otimes \cS_{H_2} = \cS_{H_1 \otimes H_2} \ ,
	\]
	for $H_1,H_2 \in \Hilb(\cC)$.
\end{remark}

\bigskip


\subsection{Bimodules and half-braidings}

In this subsection we will show that the algebraic Drinfeld center coincides with the category of $\cS$-bimodules in $\Vect(\cC^{mp} \boxtimes \cC)$.
\medskip

Through the module structure $\cC \curvearrowleft \cC^{mp} \boxtimes \cC$, algebras objects in $\Vect(\cC^{mp} \boxtimes \cC)$ can act on objects of $\Vect(\cC)$.

\begin{lemma}
    Let $H \in \Vect(\cC)$. We have
    \[ \cS_{H \triangleleft \cS} \simeq \cS_H \odot \cS \ . \]
\end{lemma}
\begin{proof}
	Observe
	\begin{align*}
		\cS_{H \triangleleft \cS}   := \cS \odot (\un \boxtimes H \triangleleft \cS) 
		& \simeq \bigoplus_{a \in \Irr(\cC)} \cS \odot (\un \boxtimes \bar{a} H a)  \\
		& \overset{(!)}{\simeq} \bigoplus_{a \in \Irr(\cC)} \cS \odot (\bar{a} \boxtimes H a) \simeq \cS \odot \left( (\un \boxtimes H) \odot \cS \right)  \simeq \left( \cS \odot (\un \boxtimes H) \right) \odot \cS = \cS_H \odot \cS \ .
	\end{align*}
	Notice that in $(!)$ we have used the canonical isomorphism $\cS \odot (\un \boxtimes \bar{a}) \simeq \cS \odot (\bar{a} \boxtimes \un)$ from Lemma \ref{lem:mirrorpropertySE}.
Alternatively, in components,
    \begin{align*}
        (\cS_H \odot \cS)(b \boxtimes c) & = \bigoplus_{a,b_1,c_1,b_2} \cC(\bar{b}_1,\bar{b}b_2) \odot \cC(c,c_1 \bar{b}_2) \odot \cC(c_1,\bar{b}_1a) \odot H(a) \\
        & \simeq \bigoplus_{a,c_1,b_2} \cC(c,c_1 \bar{b}_2) \odot \cC(c_1,\bar{b}b_2a) \odot H(a)  \simeq \bigoplus_{a,b_2}\cC(c,\bar{b}b_2a \bar{b}_2) \odot H(a) \\
        & \simeq \bigoplus_{a,b_2} \cC(bc,b_2a \bar{b}_2) \odot H(a)  \overset{(*)}{\simeq} \bigoplus_{a,b_2} \cC(a, \bar{b}_2 bc b_2) \odot H(a) \simeq \bigoplus_{a,b_2,d} \cC(d,bc) \odot \cC(a,\bar{b}_2 d b_2) \odot H(a) \\
        & \overset{(*)}{\simeq} \bigoplus_{a,b_2,d} \cC(bc,d) \odot \cC(a, \bar{b}_2 d b_2) \odot H(a) \simeq \bigoplus_{d} \cS_d(b \boxtimes c) \odot (H \triangleleft \cS)(d) \ .
    \end{align*}
\end{proof}

\begin{corollary}
    For $H \in \Vect(\cC)$, a structure of a right $\cS$-module on $H$ induces a right $\cS$-module structure on $\cS_H$.
    \label{cor:frommonadmoduletoalgebramodule}
\end{corollary}

\begin{lemma}
    Let $H \in \Vect(\cC)$. A right action $\cS_H \overset{\beta}{\curvearrowleft} \cS$ induces a right action of $\cS$ on $H$.
    \label{lemma:fromalgebramoduletomonadmodule}
\end{lemma}
\begin{proof}
    In components, using the isomorphism $\cS_H \odot \cS \simeq \cS_{H \triangleleft \cS}$, the right action $\beta$ of $\cS$ on $\cS_H$ is given by
    \[ \beta_{b \boxtimes c} : \bigoplus_{a} \cS_a(b \boxtimes c) \odot (H \triangleleft \cS)(a) \to \bigoplus_a \cS_a (b \boxtimes c) \odot H(a) \ . \]
    In particular, the $(\un \boxtimes c)$-component gives
    \[ \bigoplus_a \cS_a(\un \boxtimes c) \odot (H \triangleleft \cS)(a) \simeq \bigoplus_a \cC(c,a) \odot (H \triangleleft \cS)(a) = (H \triangleleft \cS)(c) \overset{\beta_{\un \boxtimes c}}{\longrightarrow} \bigoplus_a \cS_a(\un \boxtimes c) \odot H(a) \simeq H(c) \ . \]
    That is, 
    \[ \gamma_c := \beta_{\un \boxtimes c}: (H \triangleleft \cS)(c) \to H(c) \]
    gives a right action of $\cS$ on $H$. It is straightforward to check that $\gamma$ is indeed an action.
\end{proof}

In the above, we have crucially used the canonical identification $\cS_a(\un \boxtimes c) \simeq \cC(c,a)$, a fact that underpinnes the equivalence $\cC \simeq \cM_\cS$ of right $\cC^{mp} \boxtimes \cC$-module C$^*$-categories.

\begin{lemma}
	Given $H \in \Vect(\cC)$, there is a bijection
	\[
	\{ H \curvearrowleft (-) \triangleleft S \} \simeq \{ \cS_H \curvearrowleft \cS \}
	\]
\end{lemma}

\begin{proof}
	Given a $(-) \triangleleft \cS$-module structure $\gamma$ on $H$, let $\beta_\gamma$ be the $\cS$-module structure on $\cS_H$ provided by Corollary \ref{cor:frommonadmoduletoalgebramodule}. Conversely, if a $\cS$-module structure $\beta$ is given on $\cS_H$, let $\gamma^\beta$ denote the $(-) \triangleleft \cS$-module structure on $H$ provided by Lemma \ref{lemma:fromalgebramoduletomonadmodule}. Now, starting from $\gamma: H \triangleleft \cS \to H$, by construction we have that $\gamma^{\beta_\gamma} = \gamma$. 
\end{proof}

\begin{corollary}
    There is an equivalence between right  modules in $\Vect(\cC)$ for the monad $(-) \triangleleft \cS$ and bimodules in $\Vect(\cC^{mp} \boxtimes \cC)$ for $\cS$. That is, 
    \[ \cZ \Vect(\cC) \simeq \Bim_{\Vect(\cC^{mp} \boxtimes \cC)}(\cS) \ . \]
\end{corollary}

\begin{proof}
    Observing, as we did previously, that $\cS_a(b \boxtimes c) = \cC(bc,a) = \cS_a(\un \boxtimes bc)$, it follows that a right $\cS$-action $\beta$ on $\cS_H$ is completely determined by the components
    $\gamma_c = \beta_{\un \boxtimes c}$, using the notation above. This establishes a one-to-one correspondence between $(-) \triangleleft \cS$ module structures on an object $H \in \Vect(\cC)$ and right  $\cS$-module structures on $\cS_H$.
\end{proof}

Let $H \in \Vect(\cC)$ be equipped with an action of the monad $(-) \triangleleft \cS$. Then $\cS_H$ inherits a right $\cS$-module structure $\cS_H \odot \cS \to \cS_H$. From Lemma \ref{lemma:correspondenceleftmoduleandalgebramorphism}, the right $\cS$-module structure corresponds to an algebra morphism $\alpha: \cS \to  \cL^r_{\cS_H} := \Vect(\cC^{mp} \boxtimes \cC)(\cS_H \odot (-), \cS_H)$. Our goal now is to recover the half-braiding on $H$ from $\alpha$, and later use this to characterize unitary half-braidings. 

Since $\cS( a\boxtimes b) \simeq \delta_{a, \bar{b}} \C$, $\alpha$ is determined by algebra homomorphisms
\[ \alpha_a : \cS(\bar{a} \boxtimes a) = \C \iota_a \to \Vect(\cC^{mp} \boxtimes \cC)(\cS_H \odot (\bar{a} \boxtimes a), \cS_H)  = \cL_{\cS_H}^r(\bar{a} \boxtimes a)\ .\]

\begin{lemma}
	There is a linear ismorphism
	\[
	\Vect(\cC^{mp} \boxtimes \cC) (\cS_H \odot (\bar{a} \boxtimes a), \cS_H) \simeq \Vect(\cC^{mp} \boxtimes \cC)((H \odot a) \circ \tau, (a \odot H) \circ \tau ) \ , 
	\]
	where $\tau: \cC^{mp} \boxtimes \cC \to \cC$ is the tensor product functor.
\end{lemma}

\begin{proof}
	\begin{align*}
		\Vect(\cC^{mp} \boxtimes \cC) (\cS_H \odot (\bar{a} \boxtimes a), \cS_H) & \simeq \Vect(\cC^{mp} \boxtimes \cC) (\cS_H \odot (\un \boxtimes a), \cS_H\odot (a \boxtimes \un)) \\
		& \simeq \prod_{b,c \Irr(\cC)} \cL (\cS_H ((b\boxtimes c) \odot (\un \boxtimes \bar{a})), \cS_H((b \boxtimes c) \odot (\bar{a} \boxtimes \un))) \\
		& \simeq \prod_{b,c \in \Irr(\cC)} \cL (\cS_H(b \boxtimes c\bar{a}), \cS_H(\bar{a} b \boxtimes c)) \\
		& \simeq \prod_{b,c \in \Irr(\cC)} \cL(H(bc\bar{a}),H(\bar{a}bc)) \\
		& \simeq \prod_{b,c \in \Irr(\cC)} \cL((H \odot a)(bc), (a \odot H)(bc)) \\
		& \simeq \Vect(\cC^{mp} \boxtimes \cC)((H \odot a) \circ \tau, (a \odot H) \circ \tau ) \ .	
	\end{align*}
\end{proof}

There is a restriction map
\[ 
\Vect(\cC^{mp} \boxtimes \cC)((H \odot a) \circ \tau, (a \odot H) \circ \tau ) \to \Vect(\cC) (H \odot a, a \odot H ) \ ,
\]
induced by the embedding $\cC \hookrightarrow \cC^{mp} \boxtimes \cC$. 
\begin{lemma}
	The image of an algebra morphism $\alpha: \cS \to \cL^r_{\cS_H}$ under the restriction to $\Vect(H \odot (-), (-) \odot H)$ is a half-braiding on $H$. 
\end{lemma}
\begin{proof}
	Since the fibers of $\cS$ at $\bar{a} \boxtimes a$ are one dimensional and $\alpha$ is a unital algebra object morphism, the images $\beta_a$ of $\alpha_a$ under the restriction functor are invertible.  To check that $\beta$ is indeed a half-braiding, we have to check commutativity of the diagram
	\begin{center}
		\begin{tikzcd}
			\cS(\bar{a} \boxtimes a) \odot \cS(\bar{b} \boxtimes b) \arrow[d, "\alpha"'] \arrow[rr, "\mu^S"]                                 &  & \cS(\bar{b} \bar{a} \boxtimes a b) \arrow[d, "\alpha"]                 \\
			\cL^r_{\cS_H}(\bar{a} \boxtimes a) \odot \cL^r(\bar{b} \boxtimes b) \arrow[d, "\text{rest.}"'] \arrow[rr, "\mu^{\cL^r_{\cS_H}}"] &  & \cL^r_{\cS_H}(\bar{b} \bar{a} \boxtimes a b) \arrow[d, "\text{rest.}"] \\
			{\Vect(\cC)(H \odot a, a \odot H) \odot \Vect(\cC)(H \odot b, b \odot H)} \arrow[rr]                                             &  & {\Vect(\cC)(H \odot ab, ab \odot H)}                                  
		\end{tikzcd}
	\end{center}
	where the bottom-most arrow is the linear map
	\[
	f \odot g \mapsto (\id_a \odot g) \circ (f \odot \id_b) \ .
	\]
	Commutativity of the upper square is the associativity of $\alpha$. Commutativity of the lower square follows form the definition of $\mu^{\cL^r_{\cS_H}}$.
\end{proof}

An analogous computation shows that, for $H \in \Hilb(\cC)$,
\[
\bB^r_{\cS_H}(\bar{a} \boxtimes a) \simeq \Hilb(\cC^{mp} \boxtimes \cC) ((H \otimes a) \circ \tau, (a \otimes H) \circ \tau) \ .
\] 
Therefore, if $\alpha: \cS \to \bB^r_{\cS_H}$ is a unitary right $\cS$-module structure on $
\cS_H$, the $a$-component of the corresponding half-braiding on $H$ has value in $\Hilb(\cC) (H \otimes a, a \otimes H)$. This however only uses that the image of $\alpha$ lies in $\bB^r_{\cS_H}$. We want to argue that $\alpha$ being a $*$-algebra morphism implies that the half-braiding is actually unitary.

\begin{proposition}
	Let $H \in \Hilb(\cC)$ and let $\alpha: \cS \to \bB^r_{\cS_H}$ be a unitary right $\cS$-module structure on $\cS_H$. Then the corresponding half-braiding $\beta$ on $H$ is unitary.
	\label{prop:fromunitaryrightSmodulestounitaryhalfbraidings}
\end{proposition}

\begin{proof}
	Since $\alpha$ is a $*$-algebra object morphism, we have the commutative diagram
	\begin{center}
		\begin{tikzcd}
			\cS(\bar{a} \boxtimes a) \arrow[rr, "\alpha"] \arrow[d, "j"'] &  & { \Hilb(\cC^{mp} \boxtimes \cC)((H\otimes a) \circ \tau, (a \otimes H)\circ \tau)} \arrow[rr, "F"] \arrow[d, "j"] &  & {\Hilb(\cC) (H \otimes a, a \otimes H)} \arrow[d, "*"]                \\
			\cS(a \boxtimes \bar{a}) \arrow[rr, "\alpha"']                &  & { \Hilb(\cC^{mp} \boxtimes \cC)((H \otimes \bar{a})\circ \tau, (\bar{a} \otimes H) \circ \tau)} \arrow[rr, "F"']  &  & {\Hilb(\cC)(H \otimes \bar{a}, \bar{a} \otimes H) } \arrow[d, "\Phi"] \\
			&  &                                                                                                                   &  & { \Hilb(\cC)(a \otimes H, H \otimes a)}                              
		\end{tikzcd}
	\end{center}
	In the above diagram, $\Phi$ is the Frobenius reciprocity applied twice.
	Given a half-braiding $\beta$ on an object $H$, the inverse of the $a$-component $\beta_a$ is computed via
	\[
	\beta_a^{-1} = (\bar{R}_a^* \otimes \id_H \otimes \id_a) \circ (\id_a \otimes \beta_{\bar{a}} \otimes \id_a) \circ (\id_a \otimes \id_H \otimes R_a)
	\]
	This is exaclty what we get starting with $\iota_a$ at the top-left corner of the above diagram, and running it through the botton-right path.
\end{proof}

\begin{lemma}
	Let $H \in \Hilb(\cC)$ and suppose that $\beta$ is an algebraic half-braiding on $\For(H) \in \Vect(\cC)$. Then the corresponding algebra morphism $\alpha: \cS \to \cL^r_{\cS_H}$ factors through $\bB^r_{\cS_H}$ if the components $\beta_a: \For(H) \odot a \to a \odot \For(H)$ are bounded.
	\label{lemma:fromboundedhalfbraidingstoboundedalgebramorphism}
	\end{lemma}
	\begin{proof}
		Using the Frobenius reciprocity isomorphism, we can write
		\[
		\cL^r_{\cS_{\For(H)}} (\bar{a} \boxtimes a) \simeq \Vect(\cC^{mp} \boxtimes \cC) ((H \odot a) \circ \tau , (a \odot H) \circ \tau) \simeq \Vect(\cC^{mp} \boxtimes \cC)((\bar{a} \odot H \odot a) \circ \tau, H \circ \tau) \ .
		\]
		Under this identification, the $ b \boxtimes c$-component of $\alpha(\iota_a)$ is given by either of the two paths in the commutative diagram
		\begin{center}
			\begin{tikzcd}
				(\bar{a} \odot H \odot a)(bc) \arrow[rr, "\iota_a", hook] \arrow[rrd, "\id_{\bar{a}} \odot \beta_a"'] &  & (H \triangleleft \cS)(bc) \arrow[rr, "\alpha"]                        &  & H(bc) \\
				&  & (\bar{a} \odot a \odot H)(bc) \arrow[rru, "\bar{R}_a^* \odot \id_H"'] &  &      
			\end{tikzcd}
		\end{center}
		Thus, the norm of the $b \boxtimes c$-component of $\alpha(\iota_a)$ can be estimated: if $\beta_a$ is bounded, i.e. a morphism $H \otimes a \to a \otimes H$ of Hilbert space objects, 
		\[
		\| \alpha(\iota_a)_{b \boxtimes c} \| \leq \| \beta_a \| \| \bar{R}_a^* \| = d_a^{1/2} \| \beta_a\| \ .
		\]
		Since the right-hand side does not depend on $b$ or $c$, this gives a uniform bound on the norms of the components of $\alpha(\iota_a)$.
	\end{proof}

\begin{corollary}
	There is an equivalence 
	\[ \cZ\Hilb(\cC) \simeq \Bim_{\Hilb(\cC^{mp} \boxtimes \cC)}(\cS) \]
	of W$^*$-categories.
	\label{cor:equivalencehalfbraidingsunitarybimodulesoverSEalgebra}
\end{corollary}
\begin{proof}
	First, we make concrete the functor $\Bim_{\Hilb(\cC^{mp} \boxtimes \cC)}(\cS) \to \cZ \Hilb(\cC)$. Recall that every unitary left $\cS$-module is isomorphic to one of the form $\cS_H$, where $H \in \Hilb(\cC)$, and the left $\cS$ action is the free $\cS$-action. Recall also that this implements an equivalence
	\[ \LMod_{\Hilb(\cC^{mp} \boxtimes \cC)}(\cS) \simeq \Hilb(\cC) \ . \]
	From Proposition \ref{prop:fromunitaryrightSmodulestounitaryhalfbraidings}, a unitary right $\cS$-module structure $\alpha$ on $\cS_H$ induces a unitary half-braiding $\beta$ on $H$. The searched functor is then given by $(\cS_H, \alpha) \mapsto (H,\beta)$, but we have to prove that, if $f: (\cS_H, \alpha) \to (\cS_{H'}, \alpha')$ is a morphism of unitary $\cS$-bimodules, then the corresponding morphism $H \to H'$ is in fact a morphism of unitary half-braidings $(H,\beta) \to (H', \beta')$. This follows immediately from the commutativity of the diagram in the proof of Proposition \ref{prop:fromunitaryrightSmodulestounitaryhalfbraidings}.
	
	Let us now construct the functor in the converse direction $\cZ \Hilb(\cC) \to \Bim_{\Hilb(\cC^{mp} \boxtimes \cC)}(\cS)$. First, consider the composition
	\[
	\cZ \Hilb(\cC) \overset{\For}{\to} \cZ \Vect(\cC) \simeq \Bim_{\Vect\cC^{mp} \boxtimes 
	cC)}(\cS) \ ,
	\]
	mapping a unitary half-braiding $(H, \beta)$ to $(\cS_H,\alpha: \cS \to \cL^r_{\cS_H})$, the left $\cS$-action being the free one, and therefore unitary. Observe that the half-braiding constructed from $\alpha$ coincides with $\beta$. From Lemma \ref{lemma:fromboundedhalfbraidingstoboundedalgebramorphism}, the algebra object morphism $\alpha$ factors through an algebra morphism $\cS \to \bB^r_{\cS_H}$, which we still denote by $\alpha$. Again, from the proof of Proposition \ref{prop:fromunitaryrightSmodulestounitaryhalfbraidings}, one concludes that $\beta$ is unitary if and only if $\alpha: \cS \to \bB^r_{\cS_H}$ is a $*$-algebra object morphism.
\end{proof}

\begin{remark}
	\label{rem:tensorstructureofunitarySbimodules}
	Using Corollary \ref{cor:equivalencehalfbraidingsunitarybimodulesoverSEalgebra}, the W$^*$-category $\Bim_{\Hilb(\cC^{mp} \boxtimes \cC)}(\cS)$ can be equipped with a unitary tensor structure. On objects of the form $\cS_H$, $\cS_K$, where $H$ and $K$ are unitary half-bradings in $\Hilb(\cC)$, it is given by
	\[
	\cS_H \otimes \cS_K = \cS_{H \otimes K} \ .
	\]
	Since every unitary $\cS$-bimodule is isomorphic to $\cS_H$ for some unitary half-braiding $H$, there is a unique unitary tensor structure on $\Bim_{\Hilb(\cC^{mp} \boxtimes \cC)}(\cS)$ extending $\cS_H \boxtimes \cS_K = \cS_{H \otimes K}$, up to unitary monoidal equivalence.
	
	There is however a more intrinsic way to describe the tensor structure on the category of unitary $\cS$-bimodules. More generally, for W$^*$-algebra objects in unitary tensor categories, it is possible to define an analogue of Connes' fusion between {\em normal bimodules}. At this point this discussion seems too technical and also unnecessary for the applications we have in mind, so it will appear elsewhere.
\end{remark}

From the discussion above, we have the following.

\begin{corollary}
	\label{cor:equivalenceof2categories}
	There is an equivalence of W$^*$-2-categories
	
	\begin{align*}
		\left( \begin{matrix}
			\Hilb(\cC^{mp} \boxtimes \cC) & \Hilb(\cC) \\
			\Hilb(\cC^{mp}) & \cZ \Hilb(\cC)
		\end{matrix}
		\right) \simeq \left( \begin{matrix}
			\Hilb(\cC^{mp} \boxtimes \cC) & \LMod_{\Hilb(\cC^{mp} \boxtimes \cC)}(\cS) \\
			\RMod_{\Hilb(\cC^{mp} \boxtimes \cC)}(\cS) & \Bim_{\Hilb(\cC^{mp} \boxtimes \cC)}(\cS)
		\end{matrix}
		\right)
	\end{align*}
\end{corollary}


\section{$\cZ\Hilb(\cC)$-module C$^*$-categories}

Consider the 2-category $\CCat$ of C$^*$-categories: objects are small C$^*$-categories, 1-morphisms are $*$-functors, and 2-morphisms are unitary natural isomorphisms. This category was studied in detail in \cite{MR3007088}, where the following is shown.

\begin{theorem}[\cite{MR3007088}, Section 4]
	The max-tensor product $\boxtimes$ of C$^*$-categories, induced by the maximal tensor product of C$^*$-algebras, endows $\CCat$ with the structure of a symmetric model category. In particular, it is cocomplete and $\boxtimes$ is cocontinuous.
\end{theorem}

\begin{remark}
	Often we encounter C$^*$-categories that are only locally small, e.g. ind-completions of small C$^*$-categories. In our context, however, we can always pass to a small C$^*$-category which is equivalent to the given one, and that is closed with respect to all operations we need to perform.We shall make use of this observation without mentioning.
\end{remark}

Small C$^*$-tensor categories are exactly the algebras in $\CCat$. Given a C$^*$-tensor category $\cD$, left and right $\cD$-module C$^*$-categories $\cM$ and $\cN$, it follows from Dell'Ambrogio's results that the functor 
\[
\CCat \ni \cE \mapsto \Bal_\cD(\cM \boxtimes \cN, \cE) \ ,
\]
where $\Bal_\cD(\cM \boxtimes \cN, \cE)$ denotes the C$^*$-category of $\cD$-balanced unitary functors from $\cM \boxtimes \cN$ to $\cE$, is representable: there exists a C$^*$-category $\cM \underset{\cD}{\boxtimes} \cN$ such that
\[
\Bal_\cD(\cM \boxtimes \cN, \cE) \simeq \CCat(\cM \underset{\cD}{\boxtimes} \cN, \cE)
\]
as C$^*$-categories.

\begin{definition}
	The C$^*$-category $\cM \underset{\cD}{\boxtimes} \cN$, defined up to unitary isomorphism, is called the balanced tensor product of $\cM$ and $\cN$ over $\cD$.
\end{definition}

\subsection{Reduction to centrally pointed bimodules}

\begin{definition}
	A centrally pointed $\cC$-bimodule C$^*$-category is a triple $(\cM,m, \sigma)$ where $\cM$ is a $\cC$-bimodule C$^*$-category, $m$ is an object of $\cM$, and $\sigma: m \triangleleft (-) \simeq (-) \triangleright m$ is a unitary natural isomorphism between the action functors $\cC \to \cM$, satisfying the braid relations
	\[
	\sigma_{a\otimes b} = (\id_a \triangleright \sigma_b) \circ (\sigma_a \triangleleft \id_b) \ . 
	\]
\end{definition}

It was proved in \cite{hataishi2025categorical} that centrally pointed $\cC$-bimodule C$^*$-categories give rise to C$^*$-algebra objects in $\cZ \Vect(\cC)$, i.e. C$^*$-algebra objects in $\Vect(\cC)$ equipped with braidings which are compatible with the respective involutions. We call such C$^*$-algebra objects internal Yetter-Drinfled C$^*$-algebras. When $\Hilb (\cC) = \Rep(\bK)$ for a compact quantum group $\bK$, the category of internal Yetter-Drinfeld C$^*$-algebras is equivalent to the category of continuous $D\bK$-C$^*$-algebras. In this case, $\Rep (D\bK) \simeq \cZ \Rep(\bK)$, as unitarily braided C$^*$-tensor categories.
\bigskip

Let $\cM$ be an arbitrary right $\cZ \Hilb(\cC)$-module C$^*$-category. Using Corollary \ref{cor:equivalenceof2categories}, we obtain a right $\Hilb(\cC^{mp} \boxtimes \cC)$-module via
\[
\cM \mapsto \cM \underset{\cZ \Hilb(\cC)}{\boxtimes} \Hilb(\cC) \ .
\]

There is a canonical $*$-functor $\cM \to \cM \underset{\cZ \Hilb(\cC)}{\boxtimes} \Hilb(\cC)$, given by $\cM \ni m \mapsto m \boxtimes \un$. Recall that a right $\Hilb(\cC^{mp} \boxtimes \cC)$-module structure is the same as a $\cC$-bimodule structure.

\begin{proposition}
	Let $\cM$ be a $\cZ \Hilb(\cC)$-module C$^*$-category, and consider the induced $\cC$-bimodule C$^*$-category $\cM \underset{\cZ \Hilb(\cC)}{\boxtimes} \Hilb(\cC)$. For any object $m \in \cM$, the object 
	\[m \boxtimes \un \in \cM \underset{\cZ \Hilb(\cC)}{\boxtimes} \Hilb(\cC)
	\]
	has a canonical central structure.
	\label{prop:frommodulerforthecentertocentrallypointedbimodules}
\end{proposition}
\begin{proof}
	Indeed, the $\cC$-bimodule structure on $\cM \underset{\cZ \Hilb(\cC)}{\boxtimes} \Hilb(\cC)$ is defined by
	\[ a \triangleright (m \boxtimes x) \triangleleft b := m \boxtimes a \otimes x \otimes b \ ,\]
	for $a,b,x \in \cC$. Thus,
	\[
	a \triangleright (m \boxtimes \un) \simeq m \boxtimes a \simeq (m \boxtimes \un) \triangleleft a \ . 
	\]
\end{proof}

\subsection{Generalities about centrally pointed bimodule C$^*$-categories}

Let $(\cM,m)$ be a centrally pointed $\cC$-bimodule C$^*$-category. Denote by $\bA_m$ the internal Yetter-Drinfeld C$^*$-algebra object associated to the central object $m$. Moreover, since $\cM$ is a bimodule, there is the C$^*$-algebra object
\[
\cY_m := \underline{\End}_{\cC^{mp} \boxtimes \cC}(m) \ .
\]

\begin{lemma}
	There is a canonical isomorphism 
	\[
	\cY_m \simeq \cS_{\bA_m} \simeq \bigoplus_{a \in \Irr(\cC)} \cS_a \odot \bA_m(a)
	\]
	of C$^*$-algebra objects in $\Vect(\cC^{mp} \boxtimes \cC)$.
\end{lemma}

\begin{proof}
	The C$^*$-algebra object $\cY_m$ is defined as the object representing the functor $a \boxtimes b \mapsto \cM(a \triangleright m \triangleleft b,m)$, while $\bA_m$ is the object representing $a \mapsto \cM(m \triangleleft a,m)$. Observe now that
	\begin{align*}
		\Vect(\cC^{mp} \boxtimes \cC)(a \boxtimes b, \cY_m) & \simeq \cM(a \triangleright m \triangleleft b, m)  \simeq \cM(m \triangleleft (ab),m) \\
		& \simeq \Vect(\cC)(ab, \bA_m) \simeq \LMod_{\Vect(\cC^{mp} \boxtimes \cC)}(\cS)(\cS_{ab},\cS_{\bA_m}) \\
		& \simeq \LMod_{\Vect(\cC^{mp} \boxtimes \cC)}(\cS)(\cS \odot (a \boxtimes b),\cS_{\bA_m}) \simeq \Vect(\cC^{mp} \boxtimes \cC)(a \boxtimes b, \cS_{\bA_m}) \ ,
	\end{align*}
	and thus by universality we obtain the claim.
\end{proof}

In terms of $\cY_m$, the half-braiding on $\bA_m$ correspond to isomoprhism $\cY_m \odot (\un \boxtimes a) \simeq \cY_m \odot (a \boxtimes \un)$, natural in $a$ and satisfying the factorizability with respect to tensor products. Indeed,
\begin{align*}
	\cY_m \odot (\un \boxtimes a) \simeq \cS_{\bA_m} \odot (\un \boxtimes a) &\simeq \cS_{\bA_m \odot a} \simeq \cS_{a \odot \bA_m }\\
	& \simeq \cS \odot (a \boxtimes \un) \odot (\un \boxtimes \bA_m) \simeq \cS \odot (\un \boxtimes \bA_m) \odot (a \boxtimes \un) \\
	& = \cS_{\bA_m} \odot (a \boxtimes \un) \simeq \cY_m \odot (a \boxtimes \un) \ .
\end{align*}

The composition
\begin{align*}
	\cS \simeq \cS_{\un} \odot \C 1 \hookrightarrow \cS_{\un} \odot \bA_m(\un) \hookrightarrow \cS_{\bA_m} \simeq \cY_m
\end{align*}
is a $*$-algebra object monomorphism $\cS \to \cY_m$, such that
\begin{center}
	\begin{tikzcd}
		\cY_m \odot (\un \boxtimes a) \arrow[rr, "\simeq"]                &  & \cY_m \odot (a \boxtimes \un)              \\
		\cS \odot (\un \boxtimes a) \arrow[u, hook] \arrow[rr, "\simeq"'] &  & \cS\odot (a \boxtimes \un) \arrow[u, hook]
	\end{tikzcd}
\end{center}
commutes.

\subsection{Realization of Yetter-Drinfeld C$^*$-algebra objects}

Given a C$^*$-algebra $A$, we write $\fgpBim(A)$ for the category of finitely generated projective Hilbert C$^*$-bimodules over $A$. See \cite{Hataishi_Palomares_2025} for more details.

Let $\cC$ be a unitary tensor category, and suppose $F: \cC^{mp} \boxtimes \cC \to \fgpBim(A)$ is a unitary tensor functor. Every C$^*$-algebra object  $\bD$ in $\Vect(\cC^{mp} \boxtimes \cC)$ gives rise to an extension
\[
A \subset A \underset{F}{\rtimes} \bD
\]
of C$^*$-algebras, called the realization of $\bD$ along $F$, or the cross-product of $A$ by $(F,\bD)$. It comes with a canonical faithful conditional expectation $\bE: A \rtimes \bD \to A \otimes \bD(\un)$. This construction is functorial in the appropriate sense, see \cite{Hataishi_Palomares_2025}.

Every unitary tensor functor $\tilde{F}: \cC \to \fgpBim(A)$ induces a unitary tensor functor $F: \cC^{mp} \boxtimes \cC \to \fgpBim(A^{\op} \otimes A)$, where $\otimes$ is either the minimal or the maximal tensor product of C$^*$-algebras. $F$ is constructued from $\tilde{F}$ as follows. First, note that there are equivalences
\[
\fgpBim(A^{\op} \underset{\min}{\otimes} A) \simeq \fgpBim(A^{\op}) \underset{\min}{\boxtimes} \fgpBim(A) \  , \ \fgpBim(A^{\op} \underset{\max}{\otimes} A) \simeq \fgpBim(A^{\op}) \underset{\max}{\boxtimes} \fgpBim(A) \ .
\]
Indeed, letting $\otimes = \underset{\max}{\otimes}$, we have the canonical forgetful functors
\[
\fgpBim(A^{\op} \otimes A) \rightarrow \Corr(A^{\op} \otimes A) \simeq \Corr(A^{\op}) \boxtimes \Corr(A) \leftarrow \fgpBim(A^{\op}) \boxtimes \fgpBim(A)
\]
(see \cite{MR4162123}, Section 3) and it is easy to see that the essential image of $\fgpBim(A^{\op} \otimes A)$ and $\fgpBim(A^{\op}) \boxtimes \fgpBim(A)$ under the corresponding forgetful functors are equivalent. 

Given $X \in \fgpBim(A)$, let $X^\#$ be the complex conjugate vector space. Then
\[
a^{\op} \triangleright \xi^\# \triangleleft b^{\op} := (b^* \triangleright \xi \triangleleft a^*)^\#
\]
defines on $X^\#$ the structure of an $A^{\op}$-bimodule. The right linear $A^{\op}$-valued inner product is then given by
\[
\langle \xi^\# , \eta^\# \rangle := \langle \eta, \xi \rangle \ . 
\]
The functor $F: \cC^{mp} \boxtimes \cC \to \fgpBim(A^{\op} \otimes A)$ is then given by the exterior tensor product
\[
\tilde{F}(a)^\# \otimes \tilde{F}(b) \ .
\]
The tensor structure of $\tilde{F}$ induces a canonical tensor structure on $F$. 

\begin{definition}
	Let $\tilde{F}:\cC \to \fgpBim(A)$ be a unitary tensor functor. The symmetric enveloping algebra associated to $\tilde{F}$ is defined to be
	\[
	B_F:= (A^{\op} \otimes A) \underset{F}{\rtimes} \cS \ ,
	\]
	where $F: \cC^{mp} \boxtimes \cC \to \fgpBim(A^{\op} \otimes A)$ is the unitary tensor functor induced by $F$.
	\label{def:symmetricenvelopingalgebraofatensorfunctor}
\end{definition}

\begin{remark}
	The terminology {\em symmetric enveloping algebra} is reminiscent from subfactor literature, see (\cite{MR1302385},\cite{MR1742858}, \cite{MR1332979}).
\end{remark}

We have seen that to a pointed $\cZ \Hilb(\cC)$-module C$^*$-category $(\cM,m)$ it is associated a centrally pointed $\cC$-bimodule C$^*$-category, namely $(\cM \underset{\cZ \Hilb(\cC)}{\boxtimes} \Hilb(\cC), m \boxtimes \un)$. The internal endomorphism algebra $\cY_m$ of $m \boxtimes \un$ is a C$^*$-algebra object in $\Vect(\cC^{mp} \boxtimes \cC)$ contaning $\cS$ as a C$^*$-subalgebra. 

\begin{theorem}
	Given a pointed right $\cZ \Hilb(\cC)$-module C$^*$-category as above, there is a canonical inclusion
	\[
	B_F \subset (A^{\op} \otimes A) \underset{F}{\rtimes} \cY_m \ .
	\]
	Moreover, if $\psi: (\cM, m) \to (\cN,n)$ is a morphism of pointed right $\cZ \Hilb(\cC)$-module C$^*$-categories, then there is an induced $*$-homomorphism $\psi_*: (A^{\op} \otimes A) \underset{F}{\rtimes} \cY_m \to (A^{\op} \otimes A) \underset{F}{\rtimes} \cY_n$ that fixes $B_F$ point-wise. The operation $\psi \mapsto \psi_*$ respects composition, and we obtain therefore a functor from the category of pointed right $\cZ \Hilb(\cC)$-module C$^*$-categories to extensions of the symmetric enveloping algebra $B_F$.
	\label{thm:extensionsofsymmetricenvelopingalgebra}
\end{theorem}
\begin{proof}
	This is an immediate consequence of the functoriality of realization, as explained in \cite{Hataishi_Palomares_2025}.
\end{proof}

\subsection{Realization of Categories of Internal Modules}

In this section, due to our choices of conventions, we will have to make use of both right and left Hilbert C$^*$-modules. For a C$^*$-algebra $B$, denote by $\Corr^l(B)$ its category of left Hilbert C$^*$-correspondences, and continue denoting by $\Corr(B)$ its category of right Hilbert C$^*$-correspondences. 

Let $\cD$ be a unitary tensor category, $\bD$ a C$^*$-algebra object in $\Vect(\cD)$, and suppose $F: \cD \to \Corr(B)$ is a fully-faithful unitary tensor functor. Recall that $\cM_\bD$ denotes the Cauchy completion of the category of free left $\bD$-modules of the form $\bD \odot a =: \bD_a$, with $a \in \bD$.

The goal of this section is to prove the following Theorem.

\begin{theorem}
	There is a canonical fully-faithful unitary functor
	\[ | \cdot |_F : \cM_\bD \to \Corr^l(B \rtimes \bD, B) \ . \]
	\label{thm:realizationofcategoriesofinternalmodules}
\end{theorem}

Given $a \in \cD$, let
\[
| \bD_a |_F^\circ := \bigoplus_{b \in \Irr(\cD)} F(b) \odot \bD_a(b) \simeq \bigoplus_{b \in \Irr(\cD)} F(b) \odot \bD(b \bar{a}) \ .\]

Given $ \xi_c \odot \eta_c \in F(c) \odot \bD(c)$ and $\xi_b \odot \eta_b \in F(b) \odot \bD_a(b)$, define
\[
(\xi_c \odot \eta_c ) \triangleright  (\xi_b \odot \eta_b) := \sum_{d \in \Irr(\cD)} \sum_{v \in O(d,cb)} F(v^*) \left( F_2(\xi_c \boxtimes \xi_b) \right) \odot \bD(v \otimes \id_{\bar{a}}) \left( \bD_2 (\eta_c \odot \eta_b) \right)
\]

Recall that $B \underset{F}{\rtimes} \bD$ is obtained as a C$^*$-completion of the $*$-algebra
\[ \bigoplus_{c \in \Irr(\cD)} F(c) \odot \bD(c) = | \bD_\un |_F^\circ =: | \bD |_F^\circ \ .\]
See \cite{Hataishi_Palomares_2025}. Identically to how one describes the algebra structure of $| \bD |^\circ_F$, the following Lemma is proven. 

\begin{lemma}
	The above formula extends to an action of the algebra $| \bD|_F^\circ $ on $| \bD_a |_F^\circ$.
\end{lemma}

\begin{definition}
	Let $\langle \cdot , \cdot \rangle: | \bD_a|_F^{\circ} \times |\bD_a |_F^{\circ} \to |\bD|_F^\circ$ be given by
	\[
	\langle \xi_{b_1} \odot \eta_{b_1}, \xi_{b_2} \odot \eta_{b_2} \rangle := \sum_{c \in \Irr(\cD)} \sum_{v \in O(c,b_1 \bar{b}_2)} F(v^*) \left( F_2(\xi_{b_1} \odot j^F(\xi_{b_2})) \right) \odot \bD(v) \left( \bD_2 ((\eta_{b_1} \odot j^\bD(\eta_{b_2}))(\id_{b_1} \odot R_a \odot \id_{\bar{b}_2})) \right)
	\]
\end{definition}

\begin{lemma}
	The function $\langle \cdot , \cdot \rangle$ defined above is a faithful left $| \bD |_F^\circ$-sesquilinear form. 
\end{lemma}
\begin{proof}
	All properties but positivity and faithfulness are easy to obtain, argument being identical to the construction of the $*$-algebra structure of $|\bD|_F^{\circ}$. 
	
	First, observe that for $\xi \odot \eta \in F(b) \odot \bD_a(b)$ we have $\langle \xi \odot \eta , \xi \odot \eta \rangle \neq 0$, due to the equalities
	\begin{align*}
		\| F_2(\xi \odot j^F(\xi)) \| &= \| \xi \underset{B}{\boxtimes} j(\xi) \| > 0, \\
		\| \bD_2 ((\eta \odot j^\bD(\eta))(\id_{b} \odot R_a \odot \id_{\bar{b}})) \| &= d_a^{1/2} \| \eta \underset{\bD(\un)}{\boxtimes} j(\eta) \| > 0 \ .
	\end{align*}
	The first equality in the second line follows from the fact that $d_a^{-1/2} R_a$ is an isometry. Moreover, $\langle \xi \odot \eta , \xi \odot \eta \rangle$ is a positive element in $|\bD|_F^\circ$. Indeed, 
	\[
	F_2 (\xi \odot j(\xi)) \odot \bD_2 (\eta \odot j(\eta)) \in F(b\bar{b}) \odot \bD(b \bar{a} a\bar{b})
	\]
	projects to a positive element of $|\bD|_F^\circ$. Now, $\bD(\id_b \odot R_a^* \odot \id_{\bar{b}}): \bD(b\bar{a}a\bar{b}) \to \bD(b \bar{b})$ is a completely positive map, and therefore the projection of  
	\[
	F_2 (\xi \odot j(\xi)) \odot \bD_2 (\eta \odot j(\eta)(\id_b \odot R_a \odot \id_{\bar{b}})) =  (\id \odot \bD(\id_b \odot R_a^* \odot \id_{\bar{b}})) \left( F_2 (\xi \odot j(\xi)) \odot \bD_2 (\eta \odot j(\eta)) \right)
	\]
	is also positive in $|\bD|_F^\circ$.
	
	From the additivity of both $F$ and $\bD$, the claimed properties of $\langle \cdot , \cdot \rangle$ follow from the above by an application of a standard amplification trick: seeing vectors in $\oplus_b F(b) \odot \bD_a(b)$ as vectors in $F(X) \odot \bD_a(X)$, where 
	\[X = \bigoplus_{i \in I} b_i 
	\]
	for a finite indexing set $I$.
\end{proof}

\begin{definition}
	Let $|\bD_a|_F$ be the completion of $| \bD_a|_F^\circ$ as a left $B \underset{F}{\rtimes} \bD$-Hilbert C$^*$-module. 
\end{definition}

\begin{lemma}
	The right $A$-actions on the first tensor factor of $F(b) \odot \bD(b)$ induce a $*$-homomorphism $\rho: A^{\op} \to \bB^l_{B \rtimes \bD} (|\bD_a|_F)$. In other words, $|\bD_a|_F \in \Corr^l(B \rtimes \bD,A)$.
\end{lemma}
\begin{proof}
	Indeed, the right given $A$-action is a direct sum of amplifications of the original right $A$ actions on the modules $F(b)$.
\end{proof}

The following Lemma can be verified straightforwardly.
\begin{lemma}
	If $f \in \cM_\bD(\bD_{a_1}, \bD_{a_2}) \simeq \bD(a_1 \bar{a}_2)$, then the linear maps $\id_{F(b)} \odot f_b: F(b) \odot \bD_{a_1}(b) \to F(b) \odot \bD_{a_2}(b)$ induce a morphism $|f|: |\bD_{a_1}| \to |\bD_{a_2}|$ of $(B \rtimes \bD,B)$-correspondences. If $f$ and $g$ are composable, $|g \circ f| = |g| \circ |f|$. Thus, there is a well defined functor $| \cdot |_F: \cM_\bD \to \Corr^l(B \rtimes \bD, B)$.
\end{lemma}

\begin{proof}[Proof of Theorem \ref{thm:realizationofcategoriesofinternalmodules}]
   It remains to prove that, for $a_1, a_2 \in \cD$,
   \[
   \Corr^l(B \rtimes \bD, B) (|\bD_{a_1}|_F, |\bD_{a_2}|_F) \simeq \cM_{\bD}(\bD_{a_1}, \bD_{a_2}) \simeq \bD(a_1 \bar{a}_2) \ . 
   \]
   Let $T \in \Corr^l(B \rtimes \bD, B) (|\bD_{a_1}|_F, |\bD_{a_2}|_F)$. Let $\zeta \in |\bD_{a_1}|_F$ be such that $T(\zeta) \neq 0$. Then 
   \[ 
   \bE (\langle T(\zeta), T(\zeta)\rangle) > 0 
   \] 
   in $B \otimes \bD(\un)$ by faithfulness of the canonical conditional expectation. There exists a state $\omega$ on $\bD(\un)$ such that 
   \[
   (\id_B \otimes \omega) ( \bE (\langle T(\zeta), T(\zeta)\rangle)) > 0 \]
   in $B$. Consider the left $B$-$B$-correspondences $L^2_\omega(|\bD_{a_i}|_F) \simeq |(L^2_\omega \bD)_{a_i}|_F$, $i =1,2$. There are canonical linear maps
   \[
   \iota_\omega: \cM_\bD(\bD_{a_1},\bD_{a_2}) \to \LMod_{\Hilb(\cD)}((L^2_\omega \bD)_{a_1}, (L^2_\omega \bD)_{a_2}) \simeq \Corr^l(B,B)( |(L^2_\omega \bD)_{a_1}|_F,  |(L^2_\omega \bD)_{a_2}|_F)
   \]
   and
   \[
   \tilde{\iota}_\omega: \Corr^l(B \rtimes \bD, B) (|\bD_{a_1}|_F, |\bD_{a_2}|_F) \to \Corr^l(B,B)( |(L^2_\omega \bD)_{a_1}|_F,  |(L^2_\omega \bD)_{a_2}|_F)
   \]
   Then $\Im(\tilde{\iota}_\omega) \subset \Im(\iota_\omega)$. The inequality $ (\id_B \otimes \omega) ( \bE (\langle T(\zeta), T(\zeta)\rangle)) > 0$ now implies that $\tilde{\iota}_\omega(T) \neq 0$, so $T$ is a non-zero operator in $\Im(\iota_\omega) \simeq \cM_\bD(\bD_{a_1},\bD_{a_2})$.
\end{proof}

\section{Factorization Homology}

Factorization Homology is a mechanism which produces fully extended topological quantum field theories from relatively simple algebraic and categorical data. The general theory is based on $\infty$-categories, but for us only $(2,1)$-categories are needed.

Let $\Disk_2$ be the category whose objects are disjoint unions of compact framed 2-disks. Morphisms in $\Disk_2$ are framed embeddings. On the hom-space between two objects, consider the compact-open topology. Considering isotopies between framed embeddings, $\Disk_2$ becomes a $(2,1)$-category, i.e. a 2-category in which all 2-morphisms are invertible. Disjoint unions endow it with a symmetric monoidal structure (we reserve the term tensor for linear categories). As a symmetric monoidal $(2,1)$-category, it is generated by
\begin{itemize}
	\item an object $D$, the standard framed disk,
	\item a 1-morphism $\mu$, embedding $D \sqcup D$ into $D$
	\item and a 2-morphism $\beta$, depicted in the Figure below.
\end{itemize}
\begin{center}
	\begin{figure}[h]
	\includegraphics[scale=0.2]{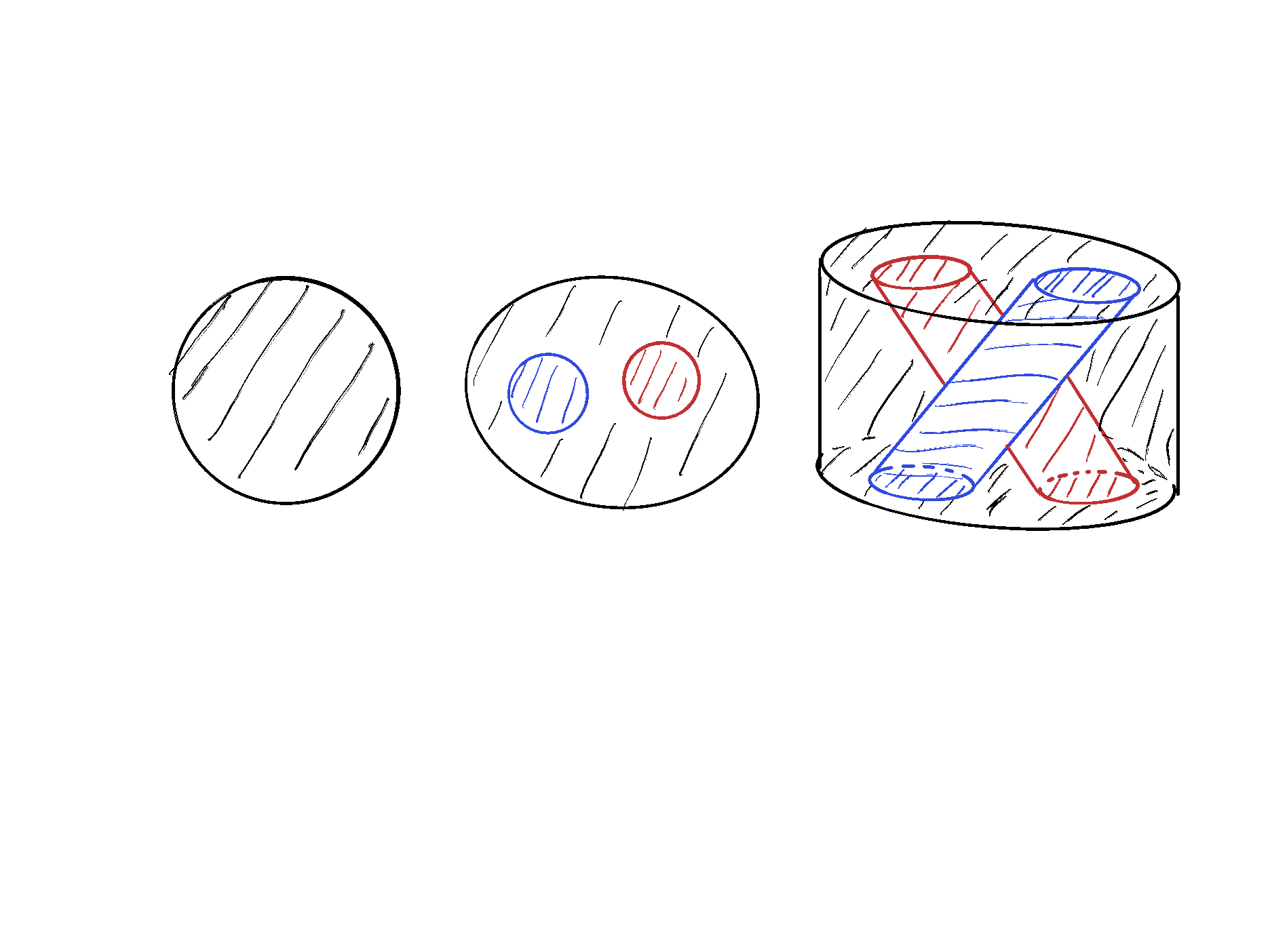}
	\label{fig:Diskcategory}
	\caption{Genereting object, 1-morphism and 2-morphism in $\Disk_2$, respectively}
\end{figure}
\end{center}

The relations are
\begin{itemize}
	\item $\mu \circ (\id_D \sqcup \mu)$ is isotopic to $\mu \circ(\mu \sqcup \id_D)$, 
	\item $\beta$ satisfies the braid relation up to isotopy.
\end{itemize}
The fact that $\Disk_2$ is generated by the data $(D,\mu,\beta)$ implies that a symmetric monoidal functor $\Disk_2 \to \cA$ is completely determined  by the image of $(D,\mu,\beta)$.  

$\Disk_2$ is a symmetric monoidal full $(2,1)$-subcategory of $\Surf$, the category of compact framed surfaces, framed embeddings and isotopies. Roughly speaking, every object in $\Surf$ is a colimit of embedded disks. Making this statement precise gives us the following result.

\begin{theorem}[\cite{MR3431668}, Theorem 1.2]
	Let $F: \Disk_2 \to \cA$ be a symmetric monoidal functor, determined by $\cD := F(D,\mu,\beta)$. If $\cA$ is cocomplete and its monoidal structure is cocontinuous, then there exists the left Kan-extension of $F$ to $\Surf$.
	\begin{center}
		\begin{tikzcd}
			\Surf \arrow[rd, "\tilde{F}", dashed]   &     \\
			\Disk_2 \arrow[r, "F"'] \arrow[u, hook] & \cA
		\end{tikzcd}
	\end{center}
	\label{thm:fundamentaltheoremofFH}
\end{theorem}

\begin{definition}
	In the context of Theorem \ref{thm:fundamentaltheoremofFH}, $\cD$ is called an $\bE_2$-algebra in $\cA$, and $\tilde{F}$ is called the {\em Factorization Homology} with coefficient in $\cD = F(D,\mu,\beta)$, and is denoted by
	\[
	\int_{(-)} \cD : \Surf \to \cA  \ .
	\]
	\label{def:FH}
\end{definition}

Factorization homology satisfies an important condition, which is usually what allows concrete computations. Let $\Sigma,\Sigma_1,\Sigma_2 \in \Surf$, let $S$ be a compact framed 1-dimensional manifold, and let $I$ be the unit interval with standard framing. Suppose that $\Sigma$ is the collar-gluing of $\Sigma_1$ with $\Sigma_2$ along $S$:
\[
\Sigma \simeq \Sigma_1 \underset{S \times I} {\sqcup} \Sigma_2 \ .
\]
From functoriality of factorization homolory, it follows that $\int_{S \times I} \cD$ is an algebra in $\cA$, and that $\int_{\Sigma_1} \cD$ and $\int_{\Sigma_2} \cD$ are right and left modules for it, respectively. 

\begin{theorem}[\cite{MR3431668}, Theorem 1.2]
	Denote the monoidal structure of $\cA$ by $\boxtimes$. In the above context,
	\[
	\int_\Sigma \cD \simeq \int_{\Sigma_1} \cD \underset{\int_{S \times I} \cD}{\boxtimes} \int_{\Sigma_2} \cD \ .
	\]
	We say that $\int_{(-)} \cD$ satisfies excision with respect to collar-gluing.
\end{theorem}

Taking $\cA$ to be a suitable symmetric monoidal $(2,1)$-category of linear categories, an $\bE_2$-algebra in $\cA$ is the same as a braided tensor category in $\cA$. In particular, the representation categories of quasi-triangular quantum groups are examples of $\bE_2$-algebras. In this context \cite{MR3847209} and \cite{MR3874702} relate factorization homology with geometric representation theory and quantization of topological gauge theories. 

We consider now running factorization homology with categories of unitary representations of locally compact quantum groups, and more generally C$^*$-tensor categories, as coefficients. For this, we take $\cA$ to be the $(2,1)$-category $\CCat$. It is then a simple exercise to show that $\bE_2$-algebras in $\CCat$ are small unitarily braided C$^*$-tensor categories.  With the above discussion, we conclude the following fact.

\begin{theorem}
	Let $\cC$ be a unitary tensor category, and $\cZ \Hilb(\cC)$ its unitary Drinfeld center. There is a unique symmetric monoidal functor
	\[
	\int_{(-)} \cZ \Hilb(\cC): \Surf \to \CCat \ , 
	\]
	factorization homology with coefficients in $\cZ \Hilb(\cC)$, such that
	\[ \int_D \cZ \Hilb(\cC) = \cZ \Hilb(\cC) \ , \]
	and satisfying excision with respect to collar-gluing:
	\[
	\Sigma \simeq \Sigma_1 \underset{S \times I} {\sqcup} \Sigma_2 \implies \int_{\Sigma} \cZ\Hilb(\cC) \simeq \int_{\Sigma_1} \cZ\Hilb(\cC) \underset{\int_{S \times I}\cZ \Hilb(\cC)} {\boxtimes} \int_{\Sigma_2} \cZ\Hilb(\cC) \ . 
	\]
		
\end{theorem}
\bigskip

Suppose the framed surface $\Sigma$ has a boundary component $S$. By thickening $S$, we obtain an embedding $S\times I \to \Sigma$. Outside $S \times I$, a copy of $\Sigma$ can be embedded, and inside the collar $S \times I$ a disk can be embedded. In summary, fixing a collar $S \times I$ induces an embeding $\Sigma \sqcup D \to \Sigma$. By functoriality, there is an induced functor
\[
\int_\Sigma \cZ \Hilb(\cC) \boxtimes \cZ \Hilb(\cC) \to \int_\Sigma \cZ \Hilb(\cC) \ .
\]

This gives to $\int_\Sigma \cZ \Hilb(\cC)$ the structure of a $\cZ \Hilb(\cC)$-module C$^*$-category. More generally, if $\Sigma$ has $n$ boundary components then factorization homology gives a $\cZ \Hilb(\cC)^{\boxtimes n}$-module C$^*$-category.
\smallskip 

The monoidal unit of $\Surf$ is the empty framed surface $\emptyset$. Its image under factorization homology is a monoidal unit of $\CCat$, which up to isomorphism is the category $\Hilb$ of Hilbert spaces. The tautological embedding $\emptyset \to \Sigma$ induces a $*$-functor
\[
\int_\emptyset \cZ \Hilb(\cC) \simeq \Hilb \to \int_\Sigma \cZ \Hilb(\cC) \ .
\]

\begin{definition}
	The image of $\C$ under the canonical functor $\Hilb \to \int_\Sigma \cZ \Hilb(\cC)$ is denoted by $\cO_\Sigma$, and is called the {\em quantum structure sheaf}.
\end{definition}

\begin{corollary}
	Suppose $\partial \Sigma \neq \emptyset$, with $n$-boundary components. Consider the induced $\cZ \Hilb(\cC)^{\boxtimes n}$-module C$^*$-category structure on $\int_\Sigma \cZ \Hilb(\cC)$. Then
	\[
	\left(  \int_{\Sigma} \cZ \Hilb(\cC) \underset{\cZ \Hilb(\cC)^{\boxtimes n}}{\boxtimes} \Hilb(\cC)^{\boxtimes n}, \cO_\Sigma \boxtimes \un  \right)
	\]
	has a canonical structure of a centrally pointed $\cC^{\boxtimes n}$-bimodule C$^*$-category. Define
	\[
	\cY_\Sigma := \underline{\End}_{(\cC^{mp} \boxtimes \cC)^{\boxtimes n}}(\cO_\Sigma \boxtimes \un) \ . 
	\]
\end{corollary}

Observe that for $\Sigma = D$ the quantum structure sheaf $\cO_D$ is just the tensor unit of $\cZ \Hilb(\cC)$, and
\[
	\left(  \int_{D} \cZ \Hilb(\cC) \underset{\cZ \Hilb(\cC)}{\boxtimes} \Hilb(\cC), \cO_D \boxtimes \un  \right) \simeq (\Hilb(\cC), \un) \ ,
\]
as a pointed right $\cC^{mp} \boxtimes \cC$-module C$^*$-category. In particular, $\cY_D \simeq \cS$ is the canonical C$^*$-algebra object.

Suppose now that $\tilde{F}: \cC \to \fgpBim(A)$ is a unitary tensor functor. From Theorem \ref{thm:extensionsofsymmetricenvelopingalgebra}, we obtain the following.
\begin{corollary}
	If $\Sigma$ is a compact framed surface with $n$-boundary components, there is an associated C$^*$-algebraic extension
	\[
	B_F^{\otimes n} \subset (A^{\op} \otimes A)^{\otimes n} \underset{F}{\rtimes} \cY_\Sigma \ .
	\]
	If $\Gamma(\Sigma)$ is the modular group of $\Sigma$ relative to its boundary, there is a canonical action of $\Gamma(\Sigma)$ by $*$-automorphisms on $(A^{\op} \otimes A)^{\otimes n} \underset{F}{\rtimes} \cY_\Sigma$ such that
	\[
	B_F^{\otimes n} \subset \left( (A^{\op} \otimes A)^{\otimes n} \underset{F}{\rtimes} \cY_\Sigma \right)^{\Gamma(\Sigma)} \ . 
	\]
	Moreover, there is a fully-faithful unitary $(\cC^{mp} \boxtimes \cC)^{\boxtimes n}$-module functor
	\[
		\left(  \int_{\Sigma} \cZ \Hilb(\cC) \underset{\cZ \Hilb(\cC)^{\boxtimes n}}{\boxtimes} \Hilb(\cC)^{\boxtimes n}, \cO_\Sigma \boxtimes \un  \right) \to \left( \Corr^l((A^{\op} \otimes A)^{\otimes n} \underset{F}{\rtimes} \cY_\Sigma, (A^{\op} \otimes A)^{\otimes n}) , (A^{\op} \otimes A)^{\otimes n} \underset{F}{\rtimes} \cY_\Sigma \right) \ .
	\]
	\label{cor:extensionofSEfromFH}
\end{corollary}

\subsection{Complex Quantum Groups}

If $\bK$ is as compact quantum group, the category $\fdRep(\bK)$ of finite dimensional unitary representations of $\bK$ is a unitary tensor category, and $\cZ \Hilb(\fdRep(\bK))$ is the category of unitary representations of the Drinfeld double $D \bK$ of $\bK$ (\cite{MR3204665}, \cite{MR3509018}). The Drinfeld double construction on compact quantum groups is the non-commutative analogue of {\em complexification}, \cite{MR4162277}.

As was already mentioned, it was proven in \cite{hataishi2025categorical} that the category of centrally cyclic $\cC$-bimodule C$^*$-categories is equivalent to the category of continuous unital $D \bK$-C$^*$-algebras, or equivalent the category of continuous unital Yetter-Drinfeld $\bK$-C$^*$-algebras. Let us denote the latter category by $\YD(\bK)$.

\begin{corollary}
	Let $\bK$ be a compact quantum group. There exists a functor
	\[
	\left\lbrace \text{ pointed $\Rep(D\bK)$-module C$^*$-categories } \right\rbrace  \to \YD(\bK) \ ,
	\]
	obtained by first
	\[
	(\cM,m) \mapsto (\cM \underset{\Rep(D\bK)}{\boxtimes} \Rep(\bK) , m \boxtimes \un) \ ,
	\]
	and then applying Theorem 4.2 in \cite{hataishi2025categorical}.
	\label{cor:fromRepDKmodulestoYetterDrinfeldlagebras}
\end{corollary}

\begin{corollary}
	Let $\bK$ be a compact quantum group. Given a compact framed surface $\Sigma$ with $n >0$ boundary components, there exists a Yetter-Drinfeld $\bK^{\times n}$-C$^*$-algebra $A_\Sigma$ such that
	\[
	\int_\Sigma \Rep(D\bK) \underset{\Rep(D\bK)^{\boxtimes n}}{\boxtimes} \Rep(\bK)^{\boxtimes n} \simeq \Hilb_{A_\Sigma}^{\bK} \ ,
	\]
	as $\Rep(\bK^{\times n})$-bimodule C$^*$-categories. There is a canonical action of $\Gamma(\Sigma)$ on $A_\Sigma$ by $D\bK^{\times n}$-equivariant $*$-automorphisms.
	\label{cor:DrinfelddoubleactionsfromFH}
\end{corollary}


\raggedright

\end{document}